\documentclass[aop, preprint]{imsart}
\usepackage{amsmath}
\usepackage{amsrefs}
\usepackage{mathrsfs}
\usepackage{amsfonts}
\usepackage{amssymb}
\usepackage{amsthm}
\usepackage[cp1250]{inputenc}
\usepackage[T1]{fontenc}
\usepackage{geometry}
\usepackage{fancyhdr}
\usepackage{graphicx}
\usepackage{color}
\usepackage{enumerate}
\usepackage{hyperref}
\usepackage{ upgreek }
\usepackage{dsfont}

\startlocaldefs

\newtheorem{defi}{Definition}[section]
\newtheorem{rem}[defi]{Remark}
\newtheorem{thm}[defi]{Theorem}

\newtheorem{cor}[defi]{Corrolary}
\newtheorem{lemma}[defi]{Lemma}
\newtheorem{fact}[defi]{Fact}

\numberwithin{equation}{section}

\newcommand{\n}{\in\mathbb{N}}

\newcommand{\li}{\lim\limits_{n\to\infty}}

\newcommand{\gr}[2]{\lim\limits_{#1\to #2}}

\newcommand{\eps}{\varepsilon}

\newcommand{\indyk}{ \mathds{1}  }
\newcommand{\ddp}[2]{\langle#1,#2\rangle}
\newcommand{\rbr}[1]{\left(#1\right)}
\newcommand{\mbr}[1]{\left|#1\right|}
\newcommand{\kbr}[1]{\left[#1\right]}
\newcommand{\cbr}[1]{\left\{#1\right\}}
\newcommand{\nbr}[1]{\left\lVert#1\right\rVert}
\newcommand{\pb}{\frac{1}{1+\beta}}

\newcommand{\E}{\mathbb{E}}
\newcommand{\R}{\mathbb{R}}

\newcommand{\Z}{\mathbb{Z}}

\newcommand{\PR}{\mathbb{P}}
\newcommand{\PG}{\mathcal{T}}

\newcommand{\pspace}{\mathcal{P}}

\newcommand{\suma}[3]{\displaystyle\sum\limits_{#1=#2}^{#3}}
\newcommand{\prd}[3]{\displaystyle\prod\limits_{#1=#2}^{#3}}
\newcommand{\calka}[2]{\int_{#1}^{#2}}

\endlocaldefs

\begin{document}

\begin{frontmatter}

\title{CLT for supercritical branching processes with heavy-tailed branching law}
\runtitle{Heavy-tailed branching CLT}

\author{\fnms{Rafa\l} \snm{Marks}\ead[label=e1]{r.marks@mimuw.edu.pl}}
\address{R. Marks\\
Institute of mathematics\\
University of Warsaw,\\
Banacha 2, 02-097 Warsaw, Poland\\
\printead{e1}}
\and
\author{\fnms{Piotr} \snm{Mi\l o\'s}\ead[label=e2]{pmilos@mimuw.edu.pl}
\ead[label=u1,url]{https://www.mimuw.edu.pl/~pmilos}}
\address{P. Mi\l o\'s \\
Institute of mathematics\\
University of Warsaw,\\
Banacha 2, 02-097 Warsaw, Poland\\
\printead{e2}\\
\printead{u1} }
\affiliation{University of Warsaw}

\runauthor{R. Marks and P. Mi\l o\'s}

\begin{abstract}
Consider a branching system with particles moving according to an Ornstein-Uhlenbeck process with drift $\mu>0$ and branching according to a law in the domain of attraction of the $(1+\beta)$-stable distribution. The mean of the branching law is strictly larger than $1$ implying that the system is supercritical and the total number of particles grows exponentially at some rate $\lambda>0$.

It is known that the system obeys a law of large numbers, see \cite{Englander2010}. In the paper we study its rate of convergence.

We discover an interesting interplay between the branching rate $\lambda$ and the drift parameter $\mu$. There are three regimes of the second order behavior:\\
$\cdot$ small branching intensity, $\lambda <(1+1/\beta) \mu$, then the speed of convergence is the same as in the stable central limit theorem but the limit is affected by the dependence between particles\\
$\cdot$ critical branching intensity, $\lambda =(1+1/\beta) \mu$, then the dependence becomes strong enough to make the rate of convergence slightly smaller, yet the qualitative behaviour still resembles the stable central limit theorem\\
$\cdot$ large branching intensity, $\lambda > (1+1/\beta) \mu$, then the dependence manifests itself much more profoundly, the rate of convergence is substantially smaller and strangely the limit holds a.s.

Our work parallels and  completes the investigations in \cite{Adamczak2015} and \cite{Ren2014} for the branching laws with finite variance. 

The case of branching laws without the second moment required developing new methods.
\end{abstract}



\end{frontmatter}

\section*{Introduction}
The law of large numbers presented in \cite{Englander2010} reveals that the empirical spatial distributions of particles of a supercritical branching system in many aspects resembles that of independent random variables. The dependence manifests when studying the second order. In \cite{Adamczak2015} authors discovered that in, what they call, the small branching regime the effect of the dependence is subtle. A CLT akin to the standard one holds but with a non-standard variance of the limit. The situation changes qualitatively in the so-called large branching regime. The normalization needs to be substantially larger, the limit is non-Gaussian and strikingly the convergence holds in $L^2$. The picture has been further refined and completed in \cite{Ren2014}.

The results described above hold for systems with the branching law having finite variance. The aim of this paper is to investigate a branching law in the domain of attraction of the stable law.

We pass to a formal description of our results. Consider a system of particles $\{X_t\}_{t\geq  0}$ initiated at time $t=0$ with a single particle located at $x\in\R^d$. This particle moves according to an Ornstein-Uhlenbeck process in $\R^d$. After exponential time with parameter $a>0$ the particle dies producing a random number of offspring. This number is distributed according to a branching law $\{p_n\}_{n\geq 0}$ whose generating function is
\begin{equation}\label{eq:generating_function}
	F(s)=ms-(m-1)+(m-1)(1-s)^{1+\beta},
\end{equation}
where $\beta\in(0,1), m>1$. Starting from the parent location the offspring evolve according to the same dynamics, moving and producing new particles. The movement processes, lifetimes and numbers of offspring are independent.

We recall that an Ornstein-Uhlenbeck is a stochastic process satisfying the stochastic differential equation $dY_t=\sigma  dB_t-\mu Y_t dt$, where $B$ is the standard Brownian motion in $\R^d$ and $\sigma>0,\mu>0$. The branching law is in the domain of attraction of a $(1+\beta)$-stable law. Importantly, we assume that its mean  $m=F'(1)$ is strictly bigger than $1$. Thus the system is supercritical, which in particular means that the total number of particles, denoted by $|X_t|$, grows exponentially as $e^{\lambda t}$, where $\lambda = a(m-1)$.

We need a functional space of functions of polynomial growth $\pspace$ defined as
$$\mathcal{P}=\left\{f:\R\mapsto\R:f\text{ is Borel and there exist }C,n>0\text{ s.t. }|f(x)|\leq C(1+|x|)^{n}\text{ for }x\in\R\right\}.$$
By $\cbr{X_t(u)}_{u=1}^{|X_t|}$ we denote positions of particles at time $t$ (we do not care about their ordering). Further, for a function $f\in \pspace$ let
$$\ddp{X_t}{f}=\suma{u}{1}{|X_t|} f(X_t(u)).$$
The aforementioned law of large numbers was stated in \cite[Theorem 6]{Englander2010} as follows
$$\frac{\ddp{X_t}{f}}{|X_t|}\rightarrow \ddp{\varphi}{f}\quad \hbox{a.s.}$$
where $\varphi$ is the stationary measure for the Ornstein-Uhlenbeck process (the result of \cite[Theorem 6]{Englander2010} holds for a much broader class of branching processes, however only for bounded $f$. The result for $f\in \pspace$ is stated in \cite[Theorem 3.1]{Adamczak2015}).
As already indicated, in this regime the behaviour is not affected by the fact that the particles' positions are dependent.

The principal aim of this paper is to find a speed of convergence in the above law of large numbers. Namely, we consider the limit of
$$\frac{\ddp{X_t}{f}-|X_t|\ddp{\varphi}{f}}{F_t},$$
where $F_t$ is some normalization.

Guided by the findings of \cite{Adamczak2015} we expect an interesting interplay of coarsening and smoothing. Notice that due to its local nature branching increases the spatial inequalities in the distribution of particles. Simply an area with more particles will produce more offspring. On the other hand the Ornstein-Uhlenbeck process is strongly mixing so it smooths the system. Imagine offspring of one particle, they will `forget` their parent's position exponentially fast. The speed of forgetting is quantified by $\mu$, the spectral gap of the Ornstein-Uhlenbeck process.

Resulting from this, there are three regimes of behaviour.

\begin{description}
\item[Small branching rate:] When $\lambda<\left(1+1/\beta\right)\mu$, the mixing induced by the Ornstein-Uhlenbeck process introduces a lot of independence to the system. Our result resembles the Gnedenko-Kolmogorov-type central limit theorem for i.i.d random variables. The limiting law is $(1+\beta)$-stable and the normalization is $F_t=|X_t|^{\frac{1}{1+\beta}}$ (it is more natural here to have a random normalization, we recall that $|X_t|\sim e^{\lambda t}$). The result is presented in Corollary~\ref{c:subcritical_simple}.
\item[Critical branching rate:] When $\lambda=\left(1+{1}/{\beta}\right)\mu$, the branching and mixing are almost in balance with the latter playing a slightly stronger role. The result again resembles the  Gnedenko-Kolmogorov-type theorem but the normalization $F_t=(t|X_t|)^{\frac{1}{1+\beta}}$ is slightly bigger than the standard one. The limiting law is $(1+\beta)$-stable. The result is presented in Corollary~\ref{c:critical_simple}.
\item[Large branching rate:] When $\lambda>\left(1+{1}/{\beta}\right)\mu$, the branching is so fast that, up to some extent, the whole genealogical structure is preserved in the limit (e.g. it depends on the initial position of the first particle). The normalization $F_t=e^{(\lambda-\mu)t}$ is much larger than $|X_t|^{\frac{1}{1+\beta}}\sim e^{\frac{1}{1+\beta}\lambda t}$. Perhaps surprisingly, the convergence holds almost surely and in $L^{1+\gamma}$ for $0\leq\gamma<\beta$. The result is presented in Corollary~\ref{c:supercritical_simple}.
\end{description}

The last result calls for higher order analysis, in fact also the limits in the first two cases might be $0$. This can be determined by the expansion of $f$ in the Hermite polynomials basis. The boundaries between cases change to the ones depending on the parameter $\kappa(f)$, which is the lowest non-zero coefficient in the expansion. In Theorem~\ref{t:subcritical}, Theorem~\ref{t:critical} and Theorem \ref{t:supercritical} we present results which take this into account, i.e. we find the normalizations which give non-trivial limits.

The phenomena studied in our paper are qualitatively much alike the ones discovered in \cite{Adamczak2015} and \cite{Ren2014}. Altogether they give a comprehensive picture - there are three regimes with boundaries depending on $\lambda, \mu$ and the domain of attraction. We mention also \cite{Adamczak2014}, which studying $U$-statistics related to the system gave more insight into the dependence structure. We think that analogous results hold for the infinite variance case.

The proof strategies of \cite{Adamczak2015} and \cite{Ren2014} rely heavily on $L^2$ calculations rendering them inapplicable in our infinite variance setting. In this paper we developed new methods. Interestingly, they can be also applied in the setting of \cite{Adamczak2015} and \cite{Ren2014} yielding shorter and arguably simpler proofs. Moreover, we show almost sure convergence in the large branching rate case, which is missing in \cite{Adamczak2015} and \cite{Ren2014}.

Branching processes have been an active research field for a long time. It is beyond the scope of this paper to present a general picture of the field instead we refer to the classical book~\cite{Athreya1972} and more recent~\cite{LeGall1999}. The interest in the spatial distribution of the particles dated back to the seventies \cite{Asmussen1976a},\cite{Asmussen1976b} and reemerged more recently, see \cite{Biggins1990}, \cite{Englander2006}, \cite{Englander2010}, \cite{Liu2013}, \cite{Harris2014}, \cite{Eckhoff2015}, \cite{Iksanov2015}, \cite{Iksanov2016}.

Notably, the authors of \cite{Englander2010} have been able to establish the convergence in a general setting of space dependent branching intensities. The first results concerning the corresponding speed of convergence appeared in~\cite{Adamczak2015}. The authors revealed the three regimes picture, much alike to the one presented above. These were further extended and refined in \cite{Mios2012},  \cite{Ren2014},  \cite{Ren2015}, \cite{Ren2017}, \cite{Ren2017c}, \cite{Wang2017}.

The paper is organised as follows. In Section~\ref{s:preliminaries} we prove preliminary facts which we use in further proofs. In Section~\ref{s:idea} we outline the main idea of proofs. The idea is formalised in Lemma \ref{l:crucial} and Lemma \ref{l:jointConvergence}, which are the core technical facts. We use them to prove our main results in relatively short Section~\ref{s:large}, Section~\ref{s:critical} and Section~\ref{s:small}. In last Section~\ref{s:open} we gather conclusions and open questions.

\section{Preliminaries}\label{s:preliminaries}
For a vector $x\in\R^d$ by $|x|$ we denote its Euclidean norm and by $\circ$ the standard scalar product. $(e_1,\ldots, e_d)$ denotes the standard basis of $\R^d$. Slightly abusing notation, for functions $f:\R^d\rightarrow\R,g:\R^d\rightarrow \R^d$ by $\ddp{g}{f}$ we denote a vector $(\ddp{g_1}{f},\ldots ,\ddp{g_d}{f})$.
\\
\\
By $X_{t}^x$ we denote the OU-branching system starting at time $0$ from $x\in\R^d$ and by $X_{s}^{u,t}$ - the subsystem starting from a particle $u\in X_{t}$ at time $t+s$. By $\cbr{\mathcal{F}_t}_{t\geq  0}$ we denote the natural filtration for the process $X$.

\subsection{Ornstein-Uhlenbeck semigroup and Hermite polynomials}
In this section we gather definition and results about the Ornstein-Uhlenbeck process and Hermite polynomials, which are required to state and prove general results in subsequent sections. The Ornstein-Uhlenbeck process has a generator $L$ given by
\begin{equation}\label{eq:infinitesimal_operator}
  L=\frac{\sigma^2}{2}\Delta-\mu (x\circ \nabla).
\end{equation}
By $\PG_t f(x)$ we denote the associated semigroup and recalling $\lambda$ we denote
\begin{equation}\label{e:pg}
  \PG_{t}^{\lambda}f:= e^{\lambda t}\PG_t f.
\end{equation}
The invariant measure of the Ornstein-Uhlenbeck process has density $\varphi:\R^d \rightarrow \R_+$ given by
$$\varphi(x)=\rbr{\frac{\mu}{\pi\sigma^2}}^{\frac{d}{2}}\exp\rbr{-\frac{\mu}{\sigma^2}|x|^2}.$$
The distribution of the Ornstein-Uhlenbeck process starting from $x$ at time $t$ is
\begin{equation} \label{eq:distribution_ou}
	\mathcal{N}\rbr{xe^{-\mu t}, \frac{\sigma^2}{2\mu} (1-e^{-2\mu t})\text{Id}}.
\end{equation}
From this we conclude easily that
\begin{equation}\label{e:tth}
\PG_t f(x)=\calka{\R^d}{}f\rbr{xe^{-\mu t}+y\sqrt{1-e^{-2\mu t}}} \varphi(y)dy=g_t\ast f(x_t).
\end{equation}
where $x_t=xe^{-\mu t}, g_t(x)=\rbr{\frac{\mu}{\pi\sigma_t^2}}^{\frac{d}{2}}\exp\rbr{-\frac{\mu}{\sigma_t^2}|x|^2}, \sigma_t^2=\sigma^2\rbr{1-e^{-2\mu t}}$. We need the smoothing property of the semigroup $\PG$. For $p=(p_1,p_2,\ldots ,p_d)\in\Z^d_+$ let $|p|=p_1+\ldots+p_n$, $p!=p_1\cdot\ldots\cdot p_n$ and $\partial x^p=\partial x_1^{p_1},\ldots \partial x_d^{p_d}$.
\begin{fact}\label{f:cinfty}
For any $f\in\mathcal{P}, t> 0$ we have $\PG_t f \in C^{\infty}$. Moreover, $\frac{\partial^{|p|}(\PG_t f)}{\partial x^p }\in\mathcal{P}$  for any $p\in\Z^d_+$.
\end{fact}
\begin{proof}
We fix $t\geq 0, i\in \{1, \ldots, d\}$ and using~\eqref{e:tth} compute
\begin{align*}
	\frac{\partial (\PG_t f)}{\partial x_i } &= \frac{\partial }{\partial x_i }\rbr{\int_{\R^d} g_t(x e^{-\mu t} - y) f(y) dy} = \int_{\R^d} \frac{\partial }{\partial x_i }g_t(x e^{-\mu t} - y) f(y) dy \\
	&= e^{-\mu t}\int_{\R^d} \frac{\partial g_t}{\partial x_i }(y) f(x e^{-\mu t} - y) dy.
\end{align*}
The second equality follows by the differentiation under the integral sign theorem and the fact that $\frac{\partial g_t}{\partial x_i}(x)=Q(x)\exp\rbr{-\frac{\mu}{\sigma_t^2}|x|^2}$ for some $Q \in \pspace$. Now it is easy to see that $\frac{\partial (\PG_t f)}{\partial x_i } \in \pspace$. Higher derivatives follow by induction. 
\end{proof}
We work with functions from space $\mathcal{P}$. The observation that $\mathcal{P}\subseteq L^2(\varphi)$ motivates using its Hilbert space structure. For $f_1,f_2\in L^2(\varphi)$ its scalar product $\ddp{\circ}{\circ}_\varphi$ is
\begin{equation}\label{e:l2scalarproduct}
  \ddp{f_1}{f_2}_{\varphi}=\calka{\R^d}{}f_1(x)f_2(x)\varphi(x) \text{d}x.
\end{equation}
We recall the Hermite polynomials $H_p$ in $\R^d$, namely
\begin{equation}\label{e:hermite}
H_p(x)=(-1)^{|p|}e^{|x|^2}\frac{\partial^{|p|}}{\partial x^p}\rbr{e^{-|x|^2}}.
\end{equation}
It is well-known, that the eigenvalues of $L$ are $\{-\mu k:k=0,1,2,\ldots\}$ and the corresponding eigenspaces $A_k$ are
$$A_k=\hbox{span} \cbr{h_p:|p|=k},$$
where $h_p$ is an appropriately scaled and normalized Hermite polynomial
\begin{equation}\label{e:normalizedHermite}
h_p(x)=\frac{1}{\sqrt{p!2^{|p|}}}H_p\rbr{\frac{\sqrt{\mu}}{\sigma} x}.
\end{equation}
In consequence these polynomials satisfy
$$\PG_th_p(x)=e^{-|p|\mu t}h_p(x).$$
The polynomials $\cbr{h_p(x)}_{p\in\Z^d_+}$ form an orthonormal basis of $L^2(\varphi)$. For a function $f\in L^2(\varphi)$ we have $f(x)=\sum\limits_{p\in\Z^d_+} a_ph_p(x)$,
where $a_p=\ddp{f}{h_p}_{\varphi}$. We denote the $L^2(\varphi)$-order of a function by
\begin{equation}\label{e:kappaf}
\kappa(f):=\inf\cbr{k\geq  0: \text{ There exists } p\in\Z^d_+: |p|=k\text{ such that }a_p\neq 0}.
\end{equation}
For example $\kappa(f)\geq 1$ means that $f$ is orthogonal to the $0$-eigenspace i.e. to the space of constant functions.\\

The following fact shows the rate of decay of the semigroup depending on $\kappa$
\begin{fact}\label{f:semigroup}
For $f\in\mathcal{P}$ there exists $R\in\mathcal{P}$ such that for any $t\geq 0$
\begin{equation}
|\PG_{t}f(x)|\leq e^{-\kappa(f)\mu t}R(x).
\end{equation}
\end{fact}
\begin{proof}
The proof follows closely the lines of \cite[Lemma 2.4]{Ren2015}, the only thing we need is Fact~\ref{f:cinfty} which is valid not only for continuous functions but for all $f\in\mathcal{P}$.
\end{proof}

We will also need the fact about differentiation of the Hermite polynomials.
\begin{fact}\label{f:differentiation}
Let $f\in L^2(\varphi)$ be differentiable and satisfy $\frac{\partial f}{\partial x_i}\in L^2(\varphi)$ for some $i\in\cbr{1,\ldots ,d}$. Then
$$\kappa\rbr{\frac{\partial f}{\partial x_i}}\geq \kappa(f)-1.$$
\end{fact}
\begin{proof}
 It follows from~\eqref{e:hermite} that for every $p\in\Z^d_+$ and some constant $C_p>0$ we have
$$h_p(x)\varphi(x)=C_p \frac{\partial^{|p|}\varphi }{x^p} (x).$$
We set $q_j=p_j$ for all $j\neq i$ and $q_i=p_i+1$ and compute
\begin{align*}
\left\langle \frac{\partial f}{\partial x_i }, h_p\right\rangle_{\varphi}&=\calka{\R^d}{}\frac{\partial f}{\partial x_i}(x) h_{p} (x)\varphi(x) dx= C_p\calka{\R^d}{}\frac{\partial f}{\partial x_i}(x) \frac{\partial ^{|p|}\varphi}{\partial x^p} (x)=C_p\calka{\R^d}{}f(x) \frac{\partial ^{|q|}\varphi}{\partial x^q} (x)\\
&=\frac{C_p}{C_q}\ddp{f}{h_q}_{\varphi}.
\end{align*}
Note that $|q|=|p|+1$ hence if we take $p$ such that $|p|<\kappa(f)-1$, then $|q|<\kappa(f)$ and $\ddp{\frac{\partial f}{\partial x_i}}{h_p}_\varphi=\frac{C_p}{C_q}\ddp{f}{h_q}_\varphi=0$. This concludes the proof.
\end{proof}

\subsection{Non-integer moments of branching processes}
In this section we gather various facts how to calculate moments and probabilities related to branching processes. The first fact holds for general random variables.
\begin{fact}\label{f:moment}
Let $X_1,\ldots X_n$ be independent random variables with mean $0$. Then for any $0<\gamma\leq 1$ there exists $C_{\gamma}> 0$ such that
\begin{equation}\label{e:moment}
\E|X_1+\ldots+X_n|^{1+\gamma}\leq C_{\gamma}( \E|X_1|^{1+\gamma}+\ldots+\E|X_n|^{1+\gamma}).
\end{equation}
\end{fact}
\begin{proof}
Let $\tilde{X_i}$ be an independent copy of $X_i$. By Jensen's inequality
\begin{align*}
\E|X_1+\ldots+X_n|^{1+\gamma} &=\E|X_1+\ldots +X_n-\E(\tilde{X_1}+\ldots +\tilde{X_n})|^{1+\gamma}\\
&\leq \E |(X_1-\tilde{X_1})+\ldots +(X_n-\tilde{X_n})|^{1+\gamma}.
\end{align*}
Denote $Z_i:=X_i-\tilde{X_i}$ and observe that $Z_i$ is symmetric, hence $Z_i\overset{d}{=} \varepsilon _iZ_i$, where $\PR(\varepsilon_i=1)=\PR(\varepsilon_i=-1)=\frac{1}{2}$ and $\varepsilon_i$ are independent and independent of all $Z=(Z_i)_{i=1}^n$. Using the inequality between moment for the conditional expectation we obtain
\begin{align*}
\E|Z_1+\ldots +Z_n|^{1+\gamma}&=\E\E( |\varepsilon_1Z_1+\ldots +\varepsilon_nZ_n|^{1+\gamma}|Z)\leq  \E(\E(\eps_1 Z_1+\ldots +\eps_lZ_n)^2|Z)^{\frac{1+\gamma}{2}}\\
&=\E(Z_1^2+\ldots+Z_n^2)^{\frac{1+\gamma}{2}}\leq\E(|Z_1|^{1+\gamma}+\ldots +|Z_n|^{1+\gamma}),
\end{align*}
where in the last inequality we have used an elementary inequality $(x_1+\ldots+x_n)^p\leq  x_1^p+\ldots+x_n^p$ for $x_1,\ldots,x_n\geq  0$ and $p\leq  1$. Finally, we observe that $\E |Z_i|^{1+\gamma}\leq 2^{\gamma} \E |X_i|^{1+\gamma}$, hence~\eqref{e:moment} holds with $C_{\gamma}=2^{\gamma}$.
\end{proof}

\begin{fact}\label{f:oumoment}
For any $0\leq\gamma<\beta$, $t>0$ and $R\in\mathcal{P}$ we have
$\E \mbr{\ddp{X_t}{R}}^{1+\gamma}<\infty.$
\end{fact}
\begin{proof} A random variable with generating function \eqref{eq:generating_function} has a finite moment of order $(1+\gamma)$.
From \cite[Theorem 2, Section III.6]{Athreya1972}, which relates existence of moments of the branching law with moments of the continuous-time Galton-Watson process it follows that for any $0\leq a< \beta$
\begin{equation}\label{eq:tmp1}
\E|X_t|^{1+a}<\infty.
\end{equation}
Let us define $M_t:=\max\limits_{u\leq |X_t|}|X_t(u)|$. It is clear that $|\ddp{X_t}{R}|\leq |X_t|\cdot M_t^l$ for some $l\n$. Fix $\gamma<a<\beta$, by the H\"older inequality we get
$$\E |X_t|^{1+\gamma}|M_t|^{(1+\gamma)l}\leq \left(\E |X_t|^{1+a}\right)^{\frac{1+\gamma}{1+a}}\cdot \left(\E |M_t|^{l(1+\gamma)q}\right)^{\frac{1}{q}}$$
for $q$ such that $\frac{1}{q}+\frac{1+\gamma}{1+a}=1$. Hence it is sufficient to show that $\E|M_t|^q<\infty$ for any $q\geq  0$.  We have for some $C,c>0$ and $A>0$
\begin{equation} \label{eq:tmp2}
\PR(M_t\geq  y,|X_t|\leq A)\leq A\PR(X_t(1)\geq  y)\leq CAe^{-cy^2},
\end{equation}
because $X_t(1)$ has a normal distribution with variance bounded in time (see \eqref{eq:distribution_ou}). By a union bound and the Markov inequality it follows that
\begin{equation*}
\PR(M_t\geq  y)\leq \PR(M_t\geq  y,|X_t|\leq e^y)+\PR(|X_t|>e^y)\leq Ce^ye^{-cy^2}+ e^{\lambda t} e^{-y}.
\end{equation*}
This proves that $\E|M_t|^q<\infty$ for any $q\leq  0$.
\end{proof}

\begin{fact}\label{f:killing}
For any $R\in\mathcal{P}$ and $0\leq \gamma<\beta$ there exists $R_1\in\mathcal{P}$ such that
$$\E_x|\langle X_1, R\rangle|^{1+\gamma}\leq R_1(x).$$
\end{fact}
\begin{proof}
It suffices to show the thesis for $R(y)=|y|^m$. By \eqref{eq:distribution_ou} we have $X_1^x \overset{d}{=} X_1 + x e^{-\mu}$. Hence, by the elementary inequality $|x+y|^m \leq 2^m (|x|^m +|y|^m$) valid for $x,y\in \R^d, m\geq 1$ and the Minkowski inequality we have
\begin{align*}
\nbr{\ddp{X_1^x}{R}}_{L^{1+\gamma}}&=\nbr{\suma{u}{1}{|X_1|}\mbr{X_1(u)+xe^{-\mu}}^m}_{L^{1+\gamma}} \leq \nbr{ 2^m\suma{u}{1}{|X_1|}\rbr{|X_1(u)|^m+|xe^{-\mu}|^m}}_{L^{1+\gamma}}\\
&\leq C_1  \nbr{\ddp{X_1}{|\circ|^m}}_{L^{1+\gamma}}+C_2|x|^m\nbr{|X_1|}_{L^{1+\gamma}}
\end{align*}
for some constants $C_1,C_2>0$. By Fact~\ref{f:oumoment} we conclude the proof.
\end{proof}
We let $W_t=\frac{|X_t|}{e^{\lambda t}}$, it is a well-known martingale related to Galton-Watson processes.
\begin{fact}\label{f:martingaleconvergence}
The process $W_t$ converges as $t\to\infty$ to $W_{\infty}$ in $L^{1+\gamma}$ for any $\gamma<\beta$. Moreover, for $s<t$ we have
$$\E|W_t-W_s|\leq \nbr{W_t-W_s}_{L^{1+\gamma}}\leq Ce^{-\frac{\lambda\gamma}{1+\gamma} s}\leq Ce^{-\frac{\lambda}{2} s}$$
for some $C=C(\gamma)>0$.
\end{fact}
\begin{proof}We denote $G_t:=|X_t|$.  By $\cbr{G_t(u)}_{u=1}^{G_s}$ we denote descendants at time $t$ of a particle $u$ at time $s$. Using Fact~\ref{f:moment} we calculate
	\begin{equation*}
		\E \mbr{G_t-G_se^{\lambda(t-s)}}^{1+\gamma} \leq C_\gamma \E\kbr{\suma{u}{1}{G_s}\rbr{\E\mbr{G_t(u)-e^{\lambda(t-s)}}^{1+\gamma}|G_s}}.
	\end{equation*}
Conditionally on $G_s$ we have that $G_t(u)$ is distributed as $G_{t-s}$, thus the right hand side is equal to $C \E G_s \E\mbr{G_{t-s}-e^{\lambda (t-s)}}^{1+\gamma}=C e^{\lambda s} \E\mbr{G_{t-s}-e^{\lambda (t-s)}}^{1+\gamma}$. Dividing both sides by $e^{\lambda(1+\gamma)t}$ we obtain
$$\E\mbr{W_t-W_s}^{1+\gamma}\leq C e^{-\lambda\gamma s}\E\mbr{W_{t-s}-1}^{1+\gamma}.$$
Using~\eqref{eq:tmp1} we obtain $\E\mbr{W_{n+1}-W_n}^{1+\gamma}\leq C_1 e^{-\lambda\gamma n}$, thus the martingale $\cbr{W_n}_{n\n}$ is bounded in $L^{1+\gamma}$.
The extension to continuous time is left to the reader.
\end{proof}

The next fact will let us control the probability of the event that small number of particles is present in the system. We recall that $a$ is the branching intensity, $m$ is the average number of particles produced when branching occurs and $\lambda = a(m-1)$.
\begin{fact}\label{f:magi}
Let $f:[0,\infty)\rightarrow [1,\infty)$ be a function such that $f(t)\leq e^{\lambda t}$ and $\gr{t}{\infty} f(t)=\infty$. Then there exist constants $C_1, C_2>0$ such that for any $t\geq 0$ there is
$$C_1 \exp\rbr{\frac{1}{m-1}(\ln f(t)-\lambda t)}\leq \PR(|X_t|\leq  f(t))\leq C_2 \exp \rbr{\frac{1}{m-1}(\ln f(t)-\lambda t)}.$$
\end{fact}
\begin{proof}
We denote $G_t:=|X_t|$.  The proof hinges on the observation that the event $\{G_t\leq f(t)\}$ typically occurs when the system does not branch until time around $u_t=t-\frac{\ln f(t)}{\lambda}$ and later branches at normal speed. Following this intuition we get
\begin{equation}\label{e:intuition}
\PR(G_t\leq f(t))\sim \PR(G_{u_t}=1)= e^{-a u_t}= e^{\frac{1}{m-1}\rbr{\ln f(t)-\lambda t}},
\end{equation}
where the second equality follows by observing that no branching occurs until time $u_t$. To formalise $\sim$ we need to prove two inequalities. To prove the first one we estimate
$$\PR(G_t\leq f(t))\geq \PR(G_{u_t}=1)\cdot \PR(G_t\leq f(t)|G_{u_t}=1).$$
We leave to the reader checking that by the choice of $u_t$ the second factor converges to $\PR(W\leq 1)$ and hence it is bounded from below by some constant $D>0$. For the second inequality let $M \in \R$, we define $u_{t,M}:=t-\frac{\ln f(t)}{\lambda}-M$ and $p_{t,M}:=\PR(G_t\leq f(t)|G_{u_{t,M}}=1)$. Observe that $G_t$ stochastically dominates the Galton-Watson process $\tilde{G}_t$ having the same intensity as $G_t$ and branching always into two particles. By direct calculations one can check $\frac{\PR(G_{t}=k)}{\PR(G_{t}=1)}\leq\frac{\PR(\tilde{G}_{t}\leq k)}{\PR(\tilde{G}_{t}=1)}\leq k$ and thus  $\frac{\PR(G_{t}=k)}{\PR(G_{t}=1)}\leq k$.  Using this fact together with the branching property we have
\begin{align}\label{e:beta}
\PR(G_{u_{t,M}}>1|G_t\leq f(t))&=\frac{\PR(G_{u_{t,M}}>1,G_t\leq f(t))}{\PR(G_t\leq f(t))}\nonumber\\
&=\frac{\sum\limits_{k=2}^{\infty} \PR(G_t\leq f(t)|G_{u_{t,M}}=k)\cdot\PR(G_{u_{t,M}}=k) }{\PR(G_t\leq f(t))}\nonumber\\
&\leq \frac{\sum\limits_{k=2}^{\infty} p_{t,M}^k\cdot\PR(G_{u_{t,M}}=k) }{p_{t,M}\cdot \PR(G_{u_{t,M}}=1)}=\sum\limits_{k=2}^{\infty} \frac{\PR(G_{u_{t,M}}=k)}{\PR(G_{u_{t,M}}=1)}p_{t,M}^{k-1}\nonumber\\
&\leq \sum\limits_{k=2}^{\infty} k p_{t,M}^{k-1}=p_{t,M}\frac{2-p_{t,M}}{(1-p_{t,M})^2}.
\end{align}
Using the Markov property and the martingale convergence we estimate
\begin{align*}
p_{t,M}&=\PR(G_{t-u_{t,M}}\leq f(t))=\PR(G_{\frac{\ln f(t)}{\lambda} +M} \leq f(t))
=\PR\left(\frac{G_{\left(\frac{\ln f(t)}{\lambda} +M\right)}}{e^{\lambda\left(\frac{\ln f(t)}{\lambda} +M\right)}}\leq e^{-\lambda M}\right) \\
&\leq \sup_{t\geq 0} \PR\rbr{W_{t}\leq e^{-\lambda M}}.
\end{align*}
Using the fact that $W_t \to W_\infty$ and the fact that $W_{\infty}$ has a density it is rather easy to show that the right-hand side converges to $0$ as $M\to +\infty$. Combining this with \eqref{e:beta} we choose $M>0$ such that $\PR(G_{u_{t,M}}>1|G_t\leq f(t))\leq 1/2.$
Consequently for such $M$ we have
\begin{align*}
\PR(G_{u_{t,M}}=1)&\geq  \PR(G_t \leq f(t), G_{u_{t,M}}=1)=\PR(G_t \leq f(t))\cdot \PR(G_{u_{t,M}}=1|G_t \leq f(t))\\
&\geq  \PR(G_t \leq f(t))/2.
\end{align*}

Proof of the both inequalities from the thesis is concluded using~\eqref{e:intuition}.
\end{proof}
We will need to know that the number of particles in $X$ is close to $e^{\lambda t}$ while their spatial extent is not too big. This is formalised by the event
\begin{equation}\label{e:et_event}
  \mathcal{E}_{t}(\varepsilon, M):=\cbr{|X_{t}|\in \rbr{e^{(\lambda-\varepsilon)t},e^{(\lambda+\varepsilon)t}},\underset{{u\leq |X_{t}|}}{\max}|X_{t}(u)|\leq M\cdot t}, \quad \varepsilon>0, M>0.
\end{equation}
\begin{fact}\label{f:maximumestimation}
For any $\varepsilon>0$ and sufficiently big $M>0$ there exist  $C>0$ and  $p>0$ such that
$$\PR\rbr{\mathcal{E}_{t}(\varepsilon, M)}\geq 1- Ce^{-pt}.$$
\end{fact}
\begin{proof}
By Fact~\ref{f:magi} we have
\begin{equation}\label{e:magisterka} \
\PR\rbr{|X_{t}|\leq e^{(\lambda-\varepsilon)t}}\leq C_1e^{-p_1\varepsilon t}
\end{equation}
for some constants $C_1>0, p_1>0.$
Similarly for any $\delta>0$ using the Markov inequality we get
\begin{equation}\label{e:markov}
\PR\rbr{|X_{t}|\geq  e^{(\lambda+\delta)t}}\leq e^{-\delta t}.
\end{equation}
We denote $M_t:=\max_{u\leq |X_{t}|}|X_{t}(u)|$ and using \eqref{eq:tmp2} we calculate
\begin{equation}
\PR\rbr{M_t> M\cdot t, |X_{t}|\leq e^{(\lambda+\delta)t}}
\leq e^{(\lambda+\delta)t}e^{-p_2M^2t^2} \leq e^{-p_3 t},
\end{equation}
for some $p_2, p_3>0$ and the last inequality holds for $M$ large enough. For these $M$ using~\eqref{e:markov} we get
\begin{equation}\label{e:tail}
\PR\rbr{M_t> M\cdot t}\leq e^{-p_4 t}, p_4>0.
\end{equation}
Combining~\eqref{e:magisterka} and~\eqref{e:tail} concludes the proof.
\end{proof}

\subsection{Characteristic functions}
For $f\in \mathcal{P}$ we denote the characteristic function of the branching process $X$. Let
$$w(x,t,\theta):=\E_x\exp(\ddp{X_t}{i\theta f}).$$
Recall the generating function $F$ (see \eqref{eq:generating_function}). The following fact will play an important role in our proofs
\begin{fact}\label{f:fouriertransform}
 The characteristic function $w$ satisfies an equation
$$w(x,t,\theta)=e^{-a t}\PG_te^{i\theta f}(x)+a\int\limits_0^t e^{-a(t-s)}\PG_{t-s}F(w(\circ,s,\theta))(x)\text{d}s,$$
which can be also written in a differential form
$$\frac{d}{dt}w(x,t)=Lw(x,t)+a(F(w(x,t))-w(x,t)), w(x,0,\theta)=e^{i\theta f(x)}.$$
In consequence we have a "many-to-one" formula:
\begin{equation}\label{e:manytoone}
E\ddp{X_t^x}{f}=e^{\lambda t}\PG_t f(x) = \PG_t^{\lambda} f(x).
\end{equation}
\end{fact}
\begin{proof}
We obtain the fact the same way as in Section 4.2 of \cite{Adamczak2015} but replacing $-\theta$ by $i\theta$.
\end{proof}

\textbf{Remark:} Some care is required when working with complex arguments. We define $z\rightarrow |z|^{1+\beta}$ on $D=\mathbb{C}\setminus \R_-$ using a branch of the complex logarithm which is discontinuous only along the line $\R_-$. With this definition the function $z^{1+\beta}$ is analytic on $D$.  Using the mean value theorem it is easy to show that
\begin{equation}\label{e:1+beta}
\mbr{x^{1+\beta}-y^{1+\beta}}\leq C\rbr{|x|^{\beta}+|y|^{\beta}}|x-y|,
\end{equation}
for $x, y$ such that $Re x\geq 0, Re y \geq 0$ and some $C>0$. \\
Looking at the formula for $F$ one observes that in the expression $F(w(x,t,\theta))$ appears the expression $(1-w(x,t,\theta))^{1+\beta}$. Since $w$ is the characteristic function, $Re (1-w)\geq 0$ , hence writing $(1-w)^{1+\beta}$ is justified.

The following fact is the generalisation of $\li \rbr{1+\frac{z}{n}}^n=e^z$ and will be used in context of characteristic functions
\begin{fact}\label{f:exp}
There exists a constant $C>0$ such that for sufficiently large $n\n$, any $a\in\mathbb{C}, b\geq  0$ and $\cbr{a_i}_{i=1}^n$ satisfying $|a_i^n|\leq n^{-3/4}$ and
$$\mbr{\rbr{\suma{i}{1}{n} a_i^n}-a}\leq b$$
we have
$$\mbr{\prd{i}{1}{n} (1+a_i^n) -\exp(a)}\leq e^{2|a|}\rbr{\frac{C}{\sqrt{n}}+b}.$$
\end{fact}
\begin{proof}
Let $\log$ be the principle branch of the logarithm in the ball $B\rbr{1,\frac{1}{2}}$. The assumption yields $1+a_i\in B\rbr{1,\frac{1}{2}}$.
Using the Taylor expansion of the logarithm we get $\log(1+x)=x+O(|x|^2)$ and thus since $|a_i^n|\leq n^{-3/4}$ we have
$$\mbr{\suma{i}{1}{n}\log(1+a_i^n)-\suma{i}{1}{n}a_i^n}\leq \frac{C}{\sqrt{n}}$$
for some constant $C>0$.
In consequence as $e^z$ is a $e^{2|a|}$-Lipschitz function on $\{|z|\leq 2a\}$ we obtain
\begin{align*}
\mbr{\prd{i}{1}{n} (1+a_i^n)-\exp(a)}&=\mbr{\exp\rbr{\suma{i}{1}{n}\log(1+a_i^n)}-\exp(a)}\\
&\leq \mbr{\exp\rbr{\suma{i}{1}{n}\log(1+a_i^n)}-\exp\rbr{\suma{i}{1}{n}a_i^n}}\\
&\quad+\mbr{\exp\rbr{\suma{i}{1}{n}a_i^n}-\exp(a)}\leq e^{2|a|}\rbr{\frac{C}{\sqrt{n}}+b}.
\end{align*}
\end{proof}

\section{Incremental decomposition}\label{s:idea}

\subsection{Main idea}
We fix $g \in \pspace$ and for notational convenience omit it in formulations of most of lemmas below.

This small section describes a decomposition which is fundamental to our proof. In subsequent Remark~\ref{r:intuitive} we will present intuitive picture.  For simplicity from now on we assume that $t\n$. We decompose $\ddp{X_t}{g}$ into $t$ parts corresponding to "conditional increment" on unit intervals. For $t\geq 1$ let
$$\Delta_{t}[g]:= \ddp{X_t}{g} - \E \rbr{\ddp{X_t}{g}|\mathcal{F}_{t-1}}= \ddp{X_t}{g} - \ddp{X_{t-1}}{\PG_{1}^{\lambda}g},$$
where the second equality follows by many-to-one formula \eqref{e:manytoone}. We also set $\Delta_0[g]:=\ddp{X_0}{g}=g(x)$, where $x$ is the initial position of the first particle. Clearly,
\begin{equation}\label{e:decomposition}
	\ddp{X_t}{g}=\suma{k}{0}{t} M^t_k[g], \quad \text{where } M^t_k[g]:= \Delta_{t-k}\kbr{\PG_k^{\lambda} g}.
\end{equation}
Moreover, it is easy to observe that $\Delta_{t}[g] = \sum_{u=1}^{|X_{t-1}|} \Delta_t[g;u]$, where
$$\Delta_t[g;u] := \ddp{X_{1}^{u,t-1}}g-\PG_{1}^{\lambda}g(X_{t-1}(u)),$$
recall that $X^{u,t-1}$ denotes the subsystem starting at time $t-1$ from a particle $X_{t-1}(u)$. We note that conditionally on $\mathcal{F}_{t-1}$ random variables $\{\Delta_t[g;u] \}_{ u \in \{1,\ldots, |X_{t-1}|\}}$ are centered and independent.
\begin{rem}
If $t\not\in\mathbb{N}$ then an additional term corresponding to the time interval $[[t],t]$ appears. In our proof this term can be easily handled with the same techniques.
\end{rem}
\begin{rem}\label{r:intuitive} We claim that
\begin{equation*}
M^t_k[g] \approx e^{\frac{\lambda}{1+\beta}t}e^{\frac{\lambda\beta-\kappa(g)\mu(1+\beta)}{1+\beta}k}.
\end{equation*}
This can be inferred from forthcoming Lemma \ref{l:singledelta} though is not directly visible. This  offers  insights how the three regimes described in Introduction arise and how to guess normalization $F_t$. Indeed
\begin{itemize}
\item If $\lambda\beta>\kappa(g)\mu(1+\beta)$, which is the supercritical case, then $M_k^t[g]$'s for large $k$'s dominate in $\ddp{X_t}{g}$. Intuitively this means that fluctuations happening early on propagate and have most impact. Moreover, we will check that  $\ddp{X_t}{g}\approx e^{(\lambda-\kappa(g)\mu)t}$.
\item If $\lambda\beta=\kappa(g)\mu(1+\beta)$, which is the critical case, then no $M_k^t[g]$ dominates. This is the most subtle case, once we show that $M_k^t[g]$ are close to be independent, it will follow that $\ddp{X_t}{g}\approx (te^{\lambda t})^{\frac{1}{1+\beta}}\approx (t|X_t|)^{\frac{1}{1+\beta}}$.
\item If $\lambda\beta<\kappa(g)\mu(1+\beta)$, which is the subcritical case, then $M_k^t[g]$ for small $k$'s dominate in $\ddp{X_t}{g}$. Intuitively this means that fluctuations happening close to $t$ have  most impact, while these from early stages are forgotten. Moreover, we will check that $\ddp{X_t}{g}\approx  e^{\frac{\lambda t}{1+\beta}} \approx |X_t|^{\frac{1}{1+\beta}}$.
\end{itemize}
\end{rem}

In the subsequent two sections we analyse properties of decomposition \eqref{e:decomposition}.

\subsection{Estimations in $L^{1+\gamma}$}
Our aim is to bound the norm of $\ddp{X_t}{g}$. We note that our estimate depends on $\kappa(g)$ introduced in \eqref{e:kappaf}. Roughly speaking, the bigger $\kappa(g)$ the norm is smaller.  We start with bounding terms in \eqref{e:decomposition}.

\begin{lemma}\label{l:singledelta_new}
For any $0\leq \gamma< \beta, R\in \pspace$ there exists $C>0$ such that for any  $h\in \pspace$ and $C_h, \kappa>0$ fulfilling
\begin{equation}\label{e:strange_assumption}
\forall_{t\geq 0}|\PG_t h(x)|\leq C_h R(x) e^{-\kappa \mu  t}
\end{equation}
we have
$$\nbr{M_k^t[h]} _{L^{1+\gamma}}\leq C C_h
e^{\frac{\lambda}{1+\gamma}t}e^{\frac{\lambda\gamma-\kappa\mu(1+\gamma)}{1+\gamma}k} $$
for any $k\in\cbr{0,\ldots, t}$.
\end{lemma}

\begin{proof} Let $k<t$, by Fact~\ref{f:moment} and fact that $\Delta_t[g;u]$ are centered and independent we have
\begin{align}\label{e:deltadecomposition}
\E \mbr{M_k^t[h]}^{1+\gamma} &= \E \E\rbr{\mbr{\suma{u}{1}{|X_{t-k-1}|}\Delta_{t-k}\kbr{\PG_k^{\lambda} h;u}}^{1+\gamma}|\mathcal{F}_{t-k-1}} \\
&\leq C \E \suma{u}{1}{|X_{t-k-1}|}\E\rbr{\mbr{\Delta_{t-k}\kbr{\PG_k^{\lambda} h;u}}^{1+\gamma}|\mathcal{F}_{t-k-1}}\nonumber
\end{align}
for some $C>0$.
Note that $\Delta_{t-k}\kbr{\PG_k^{\lambda} h;u}$  has the same distribution as $a(X_{t-k-1}(u))-b(X_{t-k-1}(u))$, where
$a(x):=\ddp{X_1^x}{\PG_k^{\lambda} h}$ and $b(x):=\PG_{k+1}^{\lambda} h(x).$
By~\eqref{e:strange_assumption} and Fact~\ref{f:killing} we get
\begin{align*}
\E |a(x)|^{1+\gamma}&=\E \mbr{\ddp{X_1^x}{\PG_k^{\lambda} h}}^{1+\gamma}
\leq C_h^{1+\gamma}e^{(1+\gamma)(\lambda -\kappa\mu)k}\E{|\ddp{X_1^x}{R_1}|}^{1+\gamma}\\
&\leq C_h^{1+\gamma} e^{(1+\gamma)(\lambda -\kappa\mu)k}R_2(x), \nonumber
\end{align*}
for some $R_1, R_2 \in \mathcal{P}$. Similarly $|b(x)|^{1+\gamma}= |\PG_{k+1}^{\lambda} h(x)|^{1+\gamma}\leq C_h^{1+\gamma}e^{(1+\gamma)(\lambda -\kappa\mu)k}R_3(x)$ for $R_3 \in \mathcal{P}$.
Combining these estimates together we get that
\begin{equation*}
\E\rbr{\mbr{\Delta_{t-k}\kbr{\PG_k^{\lambda} h;u}}^{1+\gamma}|\mathcal{F}_{t-k-1}} \leq C_h^{1+\gamma} e^{(1+\gamma)(\lambda -\kappa\mu)k}R_4(X_{t-k-1}(u)),
\end{equation*}
where $R_4 \in \mathcal{P}$. Putting this inequality into~\eqref{e:deltadecomposition} and using~\eqref{e:manytoone} we get
\begin{equation*}
\E \mbr{M_k^t[g]}^{1+\gamma}\leq C C_h^{1+\gamma}e^{(1+\gamma)(\lambda -\kappa\mu)k}
\E\ddp{X_{t-k-1}}{R_4}\leq C_1 C_h^{1+\gamma}e^{\lambda t}e^{(\gamma\lambda-\kappa(1+\gamma)\mu)k},
\end{equation*}
for $C_1 >0$. This concludes the proof for $k<t$. The easier case of $k=t$ is left to the reader.

\end{proof}

As a corollary we obtain the following lemma
\begin{lemma}\label{l:crucial_new}
For any $0\leq \gamma< \beta, R\in \pspace$ there exists $C>0$ such that for any  $h\in \pspace$ and $C_h, \kappa>0$ fulfilling
\begin{equation}
\forall_{t\geq 0}|\PG_t h(x)|\leq C_h R(x) e^{-\kappa \mu  t} \nonumber
\end{equation}
we have
$$\nbr{\ddp{X_t}{h}}_{L^{1+\gamma}}\leq CC_h\begin{cases}
e^{(\lambda-\kappa\mu)t} & \hbox{  if  }\gamma\lambda>\kappa(1+\gamma)\mu\\
te^{\frac{\lambda}{1+\gamma}t} & \hbox{  if  }\gamma\lambda=\kappa(1+\gamma)\mu\\
e^{\frac{\lambda}{1+\gamma}t} & \hbox{  if  }\gamma\lambda<\kappa(1+\gamma)\mu\\
\end{cases}.$$

\end{lemma}
\begin{proof}
 Using the decomposition~\eqref{e:decomposition}, Lemma~\ref{l:singledelta_new}, the triangle inequality and performing geometric-sequence calculations we obtain that for some $C, C_1>0$
\begin{align*}
\nbr{\ddp{X_t}{h}}_{L^{1+\gamma}}
&\leq \suma{k}{0}{t} \nbr{M_k^t[h]}_{L^{1+\gamma}}\leq C_1C_he^{\frac{\lambda}{1+\gamma}t}\suma{k}{0}{t} e^{\frac{\gamma\lambda-\kappa(g)(1+\gamma)\mu}{1+\gamma}k} \\
&\leq CC_h
		\begin{cases}
		e^{(\lambda-\kappa\mu)t} & \hbox{  if  }\gamma\lambda> \kappa(1+\gamma)\mu\\
		te^{\frac{\lambda}{1+\gamma}t} & \hbox{  if  }\gamma\lambda= \kappa(1+\gamma)\mu\\
		e^{\frac{\lambda}{1+\gamma}t} & \hbox{  if  }\gamma\lambda<\kappa(1+\gamma)\mu\\
		\end{cases}.
\end{align*}

\end{proof}
Putting $h=g$ we obtain crucial estimates for function $g$
\begin{lemma}\label{l:singledelta}
For any $0\leq \gamma <\beta$ there exists a constant $C>0$ such that
$$\nbr{M_k^t[g]} _{L^{1+\gamma}}\leq C
e^{\frac{\lambda}{1+\gamma}t}e^{\frac{\gamma\lambda-\kappa(g)(1+\gamma)\mu}{1+\gamma}k} $$
for any $k\in\cbr{0,\ldots, t}$.
\end{lemma}
\begin{proof}
By Fact~\ref{f:semigroup} we know that $g$ satisfies~\eqref{e:strange_assumption} with $\kappa=\kappa(g)$ and $C_g=1$. Now the fact follows applying Lemma~\ref{l:singledelta_new}.
\end{proof}
\begin{lemma}\label{l:crucial}
For any $0\leq \gamma <\beta$ there exists a constant $C>0$ such that
$$\nbr{\ddp{X_t}{g}}_{L^{1+\gamma}}\leq C\begin{cases}
e^{(\lambda-\kappa(g)\mu)t} & \hbox{  if  }\gamma\lambda>\kappa(g)(1+\gamma)\mu\\
te^{\frac{\lambda}{1+\gamma}t} & \hbox{  if  }\gamma\lambda=\kappa(g)(1+\gamma)\mu\\
e^{\frac{\lambda}{1+\gamma}t} & \hbox{  if  }\gamma\lambda<\kappa(g)(1+\gamma)\mu\\
\end{cases}.$$
\end{lemma}
\begin{proof}
By Fact~\ref{f:semigroup} we know that $g$ satisfies~\eqref{e:strange_assumption} with $\kappa=\kappa(g)$ and $C_g=1$. Now the fact follows applying Lemma~\ref{l:crucial_new}.
\end{proof}


\subsection{Estimates of characteristic functions}
Recall $M^t_k[g]$ defined in \eqref{e:decomposition}, it corresponds to the evolution of the process $X$ on time interval $[t-k-1,t-k]$ and is $\mathcal{F}_{t-k-1}$-measurable. The first aim of this section is to find an expansion of the characteristic function of $M^t_k[g]$ near $0$. This we further use to study the joint law of $M^t_k[g]$'s. We define
$$ g_k=\PG_k^{\lambda} g, \quad k\in \cbr{0,\ldots,t}.$$
We observe that $g_k \in \mathcal{P}$ and by Fact~\ref{f:semigroup} there exists $R\in \mathcal{P}$ such that
\begin{equation}\label{e:gkestimation}
|g_k(x)|\leq e^{(\lambda-\kappa(g)\mu)k}R(x).
\end{equation}
For $h\in \mathcal{P}$ and $\theta\geq 0$ we define
$$Z_h(x,\theta):=\calka{0}{1} \PG_{1-s}^{\lambda}\kbr{\lambda\rbr{\PG^{\lambda}_s(i\theta h(\cdot))}^{1+\beta}}(x) \text{d}s, \quad Z_h(x) := Z_h(x,1)$$
and observe
\begin{equation}\label{e:zhtheta}
  Z_h(x,\theta) = \theta^{1+\beta}Z_h(x).
\end{equation}
Moreover we define
\begin{equation}\label{eq:mkg}
  m_k[g]:=e^{-\lambda (k+1)}\ddp{Z_{g_{k}}}{\varphi}.
\end{equation}
It will turn out that the sum of $m_k[g]$ appear as parameter of stable distribution in our limit theorems.

As a warm-up we estimate
\begin{lemma}\label{l:zgmk}
There exists  $R \in \mathcal{P}$ such that for any $k\in \cbr{0,\ldots,t}$ we have
\begin{equation}
\mbr{Z_{g_k}(x)}\leq e^{(\lambda -\kappa(g)\mu)(1+\beta)k}R(x).\nonumber
\end{equation}
In consequence there exists a constant $C>  0$ such that for any $k\in \cbr{0,\ldots,t}$ we have
$$|m_k[g]|\leq Ce^{(\lambda\beta-\kappa(g)\mu(1+\beta))k}.$$
\end{lemma}
\begin{proof}
The assertion follows by~\eqref{e:gkestimation}.
\end{proof}
Moreover, we need an improvement of Fact~\ref{f:semigroup} for functions $Z_{g_k}$
\begin{lemma}\label{l:semigroup}
There exists $R\in\mathcal{P}$ such that for any $s\geq 0$ and  $h_k(x) = Z_{g_k}(x)-\ddp{Z_{g_{k}}}{\varphi} $,  $k\in \cbr{0,\ldots,t}$ we have
\begin{equation}
\mbr{\PG_{s}h_k(x)}\leq e^{-\mu s}e^{(\lambda-\kappa(g)\mu)(1+\beta)k}R(x).\nonumber
\end{equation}
\end{lemma}
\begin{proof} Using \eqref{e:tth} and $\ddp{h_k}{\varphi}=0$ we may write
\begin{align}\label{e:tth1}
\PG_sh_k(x)&= \calka{\R^d}{}\rbr{h_k\rbr{xe^{-\mu s}+y\sqrt{1-e^{-2\mu s}}}-h_k(y) }\varphi(y)dy\nonumber \\
&\leq \calka{\R^d}{} \sup |\nabla h_k(z)| \mbr{xe^{-\mu s}+y\sqrt{1-e^{-2\mu s}}-y}\varphi(y)dy.
\end{align}
where supremum is taken over $z$ in the interval connecting $y$ with $xe^{-\mu s}+y\sqrt{1-e^{-2\mu s}}$.
It is easy to check that for any $f\in \pspace$ we have
$$\frac{\partial}{\partial x_i} \PG_s^{\lambda} f(x)= e^{-\mu s} \PG_s^{\lambda}\rbr{\frac{\partial f}{\partial x_i}}(x).$$
We set $\tilde{g}=\PG_1^\lambda g$. By Fact~\ref{f:cinfty} we know that $\mbr{\nabla \tilde{g}}\in\mathcal{P}$ and $\kappa(\tilde{g})=\kappa(g)$.  For $i\in \{ 1,\ldots ,d \}$ we compute
\begin{equation}
\frac{\partial}{\partial x_i} Z_{g_k} (x)=
\int\limits_0^1 e^{-\mu(1-u)}\PG_{1-u}^{\lambda}
\kbr{\lambda\rbr{\frac{\partial}{\partial x_i}\rbr{\rbr{i\PG ^\lambda_{u+k-1} \tilde{g}}^{1+\beta}}}}(x) \text{d}u.\nonumber
\end{equation}
Further we calculate
$$\frac{\partial}{\partial x_i}\rbr{\rbr{i\PG ^\lambda_{u+k-1} \tilde{g}}^{1+\beta}}=(1+\beta)\rbr{i\PG ^\lambda_{u+k-1} \tilde{g}}^{\beta} ie^{-\mu(u+k-1)}\PG ^\lambda_{u+k-1} \frac{\partial  \tilde{g}}{\partial x_i}.$$
Using Fact~\ref{f:differentiation} we get that $\kappa\rbr{\frac{\partial  \tilde{g}}{\partial x_i}}\geq \kappa(\tilde{g})-1$, which combined with Fact~\ref{f:semigroup} yields
$$\mbr{\rbr{i\PG ^\lambda_{u+k-1} \tilde{g}}^{\beta}(x)}\leq R_1(x)e^{(\lambda-\kappa(g)\mu)\beta (u+k-1)},\mbr{\PG ^\lambda_{u+k-1} \frac{\partial  \tilde{g}}{\partial x_i}(x)}\leq R_1(x)e^{(\lambda-(\kappa(g)-1)\mu) (u+k-1)}$$
 for some  $R_1\in\mathcal{P}$. As a consequence for some $R_2 \in \pspace$ we get
$$\mbr{\frac{\partial}{\partial x_i} h_k(x)} = \mbr{\frac{\partial}{\partial x_i} Z_{g_k} (x)}\leq R_2(x) e^{(\lambda-\kappa(g)\mu)(1+\beta) k}.$$
Hence we get that for $z$ in the interval connecting $y$ with $xe^{-\mu s}+\sqrt{1-e^{-2\mu s}}y$ we have $\mbr{\nabla h_k(z)}\leq (R_3(x)+R_3(y))e^{(\lambda-\kappa(g)\mu)(1+\beta) k}$ for some $R_3 \in \pspace$. Observe also that\\
$\mbr{xe^{-\mu s}+y\rbr{\sqrt{1-e^{-2\mu s}}-1}}\leq e^{-\mu s}\rbr{R_4(x)+R_4(y)}$ for some $R_4\in\pspace$. Now using~\eqref{e:tth1} we get
$$\mbr{\PG_s h_k(x)}\leq \calka{\R^d}{} \mbr{R_5(x)+R_5(y)} e^{(\lambda-\kappa(g)\mu)(1+\beta) k} e^{-\mu s}\varphi(y)dy\leq
R_6(x)e^{-\mu s} e^{(\lambda-\kappa(g)\mu)(1+\beta) k},$$
for some $R_5, R_6 \in \pspace$. This finishes the proof for $k\geq 1$. The thesis for $k=0$ follows directly from Fact~\ref{f:semigroup}.
\end{proof}

As a corollary we state
\begin{lemma}\label{l:lln}
There exist $C,p>0$ such that for any  $h_k(x) = Z_{g_k}(x)-\ddp{Z_{g_{k}}}{\varphi} $,  $k\in \cbr{0,\ldots,t}$ we have
\begin{equation}\label{e:lln}
\E\mbr{\ddp{X_t}{h_k}}\leq \nbr{\ddp{X_t}{h_k}}_{L^{1+\beta/2}}\leq Ce^{(\lambda-p)t}e^{(\lambda-\kappa(g)\mu)(1+\beta) k}.
\end{equation}
Moreover, as a consequence we get that for any $\varepsilon \in (0, \lambda)$
\begin{equation}\label{eq:lln}
\E\mbr{\rbr{\frac{\ddp{X_t}{Z_{g_k}}}{|X_t|}-\ddp{Z_{g_{k}}}{\varphi}}\indyk_{|X_t|\geq e^{(\lambda-\varepsilon)t}}}\leq Ce^{-(p-\varepsilon)t} e^{(\lambda-\kappa(g)\mu)(1+\beta) k}.
\end{equation}

\end{lemma}
\begin{proof}
By Fact~\ref{f:semigroup} we know that $h_k$ satisfies~\eqref{e:strange_assumption} with $C_{h_k}=e^{(\lambda-\kappa(g)\mu)(1+\beta) k}$ and $\kappa=\kappa(g)$. We also put $\gamma= \beta/2$. Now~\eqref{e:lln} follows by Lemma~\ref{l:crucial_new}. Using \eqref{e:lln} one can obtain \eqref{eq:lln} by simple calculations.
\end{proof}

We want to approximate the characteristic function of $M_k^t[g]$. It is the sum of terms $\Delta_{t-k}[g_k;u] = a(X_{t-k-1}(u))$, where $a(x) = \ddp{X_{1}^x}{g_k}-\PG_{1}^{\lambda}g_k(x)$. In the following lemma we focus on the characteristic function of $a(x)$, which will be the key ingredient in the subsequent crucial Lemma \ref{l:jointConvergence}.

\begin{lemma}\label{l:singlesubsystem}
There exists $R\in\mathcal{P}$ and $p>1+\beta$ such that for any $k\n$ we have
$$\E e^{i\theta(\ddp{X_{1}^x}{g_k}-\PG_{1}^{\lambda}g_k(x))}=(1+Z_{g_k}(x;\theta))+err(x,\theta)$$
and $|err(x,\theta)|\leq \mbr{\theta e^{(\lambda -\kappa(g)\mu)k}}^{p}R(x)$.
\end{lemma}
\begin{proof}
We denote $\varphi(x,\theta):=\E e^{i\theta(\ddp{X_{1}^x}{g_k}-\PG_{1}^{\lambda}g_k(x))}$ and  set $w(x,t,\theta) :=\E e^{i\theta(\ddp{X_{t}^x}{g_k}}$. Using Fact~\ref{f:fouriertransform} and substituting the generating function $F(s)=ms-(m-1)+(m-1)(1-s)^{1+\beta}$ we get that
$$\frac{d}{dt}w(x,t)=Lw(x,t)+a(m-1)\left((w(x,t)-1)+(1-w(x,t))^{1+\beta}\right),\ w(x,0,\theta)=e^{i\theta g_k(x)}.$$
Recall $\lambda=a(m-1)$ and let us define $u(x,t)=1-w(x,t)$. It fulfills
\begin{equation}
\frac{d}{dt}u(x,t)=Lu(x,t)+\lambda u(x,t)-\lambda(u(x,t))^{1+\beta}, u(x,0,\theta)=1-e^{i\theta g_k(x)}.\nonumber
\end{equation}
In the integral form this writes as
\begin{equation}\label{e:u}
u(x,t,\theta)=\PG_t^{\lambda}\kbr{1-e^{i\theta g_k(\circ)}}(x)-\lambda\calka{0}{t}\PG_{t-s}^{\lambda}\kbr{u(\circ,s,\theta)^{1+\beta}}(x) \text{d}s.
\end{equation}
Using $u$ we represent $\varphi$ as follows
$$\varphi(x,\theta)=(1-u(x,1, \theta))\cdot e^{-i\theta \PG_1^{\lambda}g_k(x)}.$$
We also need
$$\tilde{v}(x,s,\theta) := \PG_s^{\lambda} \kbr{1-e^{i\theta g_k(\circ)}}(x)\hbox{ and } v(x,s,\theta) := i\theta \PG_s^{\lambda}  g_k(x).$$
Given a function $z$ we define
\begin{equation}\label{eq:I}
	I(z) := \lambda\calka{0}{1}\PG_{1-s}^{\lambda}\kbr{z(\circ,s,\theta)^{1+\beta}}(x) \text{d}s.
\end{equation}
Having set the notation we are ready to present the outline of the proof. In the outline $R_1, R_2, \ldots \in \mathcal{P}$ and they may vary from line to line. We denote $\theta_k =\theta e^{(\lambda-\kappa(g)\mu)k} $ and implicitly assume $|\theta_k|\leq 1$.

\begin{enumerate}[1.]
\item We bound $u, v$ and $\tilde{v}$ from above by $|\theta_k| R_1(x)$.
\item From step 1, \eqref{e:u} and \eqref{eq:I} we deduce that $|u-\tilde{v}| = |I(u)| \leq |\theta_k|^{1+\beta}R_2(x)$ and $|v-\tilde v|\leq  |\theta_k|^2R_2(x)$.
\item We estimate  $|I(u) - I(\tilde v)|\leq |\theta_k|^{1+2\beta} R_3(x) $, $|I(\tilde v) - I(v)|\leq  |\theta_k|^{2+\beta} R_3(x) $. Consequently also  $|I(u) - I( v)|\leq |\theta_k|^{1+2\beta} R_3(x) $.
\item We prove that $|I(u) - (1-\varphi(x,\theta))|\leq |\theta_k|^2 R_4(x)$.
\item We combine the previous steps (observing that $I(v)=Z_{g_k}$) to conclude the proof.
\end{enumerate}
Now we implement the step outlined in the sketch.
\begin{enumerate}[1.]
\item Using \eqref{e:gkestimation}, the elementary inequality  $|1-e^{iy}|\leq |y|$ valid for $y\in\R$ and fact that $|\PG_t h|\leq\PG_t|h|$ we obtain that for $s\leq 1$
\begin{equation}\label{e:(1)tv}
|\tilde{v}(x,s,\theta)|=\mbr{\PG_s^{\lambda} [1-e^{i\theta g_k(\circ)}](x)}\leq \mbr{\theta \PG_s^{\lambda} |g_k|(x)}\leq \mbr{\theta e^{(\lambda-\kappa(g)\mu)k}}R_1(x).
\end{equation}
and similarly
\begin{equation}\label{e:(1)v}
|v(x,s,\theta)|\leq \mbr{\theta e^{(\lambda-\kappa(g)\mu)k}}R_1(x).
\end{equation}
From defintions of $u,w$ we have $u = 1-w =\E_x\rbr{1-e^{i\theta \ddp{X_s}{g_k}}}$, hence
\begin{align}\label{e:(1)u}
|u(x,s,\theta)|&=\mbr{\E_x\rbr{1-e^{i\theta \ddp{X_s}{g_k}}}}\leq \E_x\mbr{1-e^{i\theta \ddp{X_s}{g_k}}}
\leq |\theta|\E_x \mbr{\ddp{X_s}{g_k}}\leq |\theta|\E_x \ddp{X_s}{|g_k|}\nonumber \\
&=|\theta|\rbr{\PG_s^{\lambda}|g_k|(x)}\leq \mbr{\theta e^{(\lambda-\kappa(g)\mu)k}}R_1(x).
\end{align}
In the last but one step we used \eqref{e:manytoone}.
\item We use \eqref{e:(1)u} to get
\begin{align}\label{e:(2)u-v}
|u(x,s,\theta)-\tilde{v}(x,s,\theta)|
&=|\lambda|\mbr{\calka{0}{1}\PG_{1-s}^{\lambda}[u(\circ,s,\theta)^{1+\beta}](x) \text{d}s} \leq |\lambda|\calka{0}{1}\PG_{1-s}^{\lambda}[\mbr{u(\circ,s,\theta)}^{1+\beta}](x) \text{d}s \nonumber\\
&\leq \mbr{\theta e^{(\lambda-\kappa(g)\mu)k}}^{1+\beta}|\lambda|\mbr{\calka{0}{1}\PG_{1-s}^{\lambda}[R_1(\circ)^{1+\beta}](x) \text{d}s}\nonumber\\
&\leq \mbr{\theta e^{(\lambda-\kappa(g)\mu)k}}^{1+\beta}R_2(x).
\end{align}
We use the elementary inequality $|1-e^{iy}-iy|\leq 2|y|^2$ valid for $y\in \R$ and \eqref{e:gkestimation} to obtain
\begin{equation}\label{e:(3)tv-v}
\mbr{\tilde{v}(x,\theta)-v(x,\theta)}=\mbr{\PG_s^{\lambda} [e^{i\theta g_k(\circ)}-i\theta g_k(\circ)](x)}\leq C|\theta|^2\PG_s^{\lambda} g_k^2(x)\leq \mbr{\theta e^{(\lambda-\kappa(g)\mu)k}}^{2}R_2(x).
\end{equation}

\item Using~inequality~\eqref{e:1+beta} and then \eqref{e:(2)u-v},~\eqref{e:(1)tv},~\eqref{e:(1)u} we calculate
\begin{align}\label{e:(2)Iu-Iv}
\mbr{I(u)-I(\tilde{v})}
&=|\lambda|\mbr{\calka{0}{1} \PG_{1-s}^{\lambda}\rbr{\left(u(\circ,s,\theta)^{1+\beta}-\tilde{v}(\circ ,s,\theta)^{1+\beta}\right)}(x) \text{d}s}\nonumber\\
&\leq C_1\mbr{\calka{0}{1} \PG_{1-s}^{\lambda}\rbr{|u-\tilde{v}|\cdot \rbr{|u|^{\beta}+|\tilde{v}|^{\beta}}}(x) \text{d}s}\leq \mbr{\theta e^{(\lambda-\kappa(g)\mu)k}}^{1+2\beta}R_3(x),
\end{align}
for some $C_1 >0$. Similarly using~\eqref{e:(3)tv-v} instead of \eqref{e:(2)u-v} and \eqref{e:(1)v} instead of \eqref{e:(1)u} one obtains
\begin{equation}\label{e:(3)Itv-Iv}
\mbr{I(v)-I(\tilde{v})}\leq 	\mbr{ \theta e^{(\lambda-\kappa(g)\mu)k}}^{2+\beta}R_3(x).
\end{equation}

\item We recall that $I(u)=u(x,\theta)-\tilde{v}(x,\theta)$, ${v}(x,\theta) = i\theta \PG_1^{\lambda}g_k(x)$ hence
\begin{align*}
\varphi(x,\theta)-1+I(u)&=(1-u(x,\theta))\cdot e^{-v(x,\theta)}-1+I(u)\\
&=\rbr{e^{-v(x,\theta)}-\tilde{v}(x,\theta)-1}+u(x,\theta)\rbr{1-e^{-v(x, \theta)}}.
\end{align*}
To deal with the first summand, we use the elementary inequality $|1-e^{-iy}+iy|\leq 2|y|^2$ valid for $y\in \R$ and then \eqref{e:(1)v}, \eqref{e:(3)tv-v} to estimate
$$|e^{-v(x,\theta)}-\tilde{v}(x,\theta)-1|\leq |e^{-v(x,\theta)}-v(x,\theta)-1|+|\tilde{v}(x,\theta)-v(x,\theta)|\leq \mbr{\theta e^{(\lambda-\kappa(g)\mu)k}}^{2}R_4(x).$$
The estimate of the second summand follows by \eqref{e:(1)u} and \eqref{e:(1)v}:
$$\mbr{u(x,\theta)\rbr{1-e^{-v(x,\theta)}}}\leq \mbr{\theta e^{(\lambda-\kappa(g)\mu)k}}^{2}R_4(x).$$
We conclude using the triangle inequality
\begin{equation}\label{e:(4)phi}
\mbr{\varphi(x,\theta)-1+I(u)}\leq \mbr{\theta e^{(\lambda-\kappa(g)\mu)k}}^{2}R_4(x).
\end{equation}
\item The previous steps prove $\varphi(x,\theta) \approx 1+Z_{g_k}(x, \theta) =  1-I(v)$. Formally,
\begin{align*}
\mbr{\varphi(x,\theta)-1-Z_{g_k}(x,\theta)}&=\mbr{\varphi(x,\theta)-1+I(v)}\\
&\leq \mbr{\varphi(x,\theta)-1+I(u)}+\mbr{I(u)-I(\tilde{v})}+\mbr{I(\tilde{v})-I(v)}\nonumber
\end{align*}
Combining~\eqref{e:(2)u-v},~\eqref{e:(3)Itv-Iv} and~\eqref{e:(4)phi} we get
\begin{equation*}
	\mbr{\varphi(x,\theta)-1-Z_{g_k}(x,\theta)} \leq R(x)|\theta e^{(\lambda-\kappa(g)\mu)k}|^p
\end{equation*}
for some $R\in \mathcal{P}$ and $p=\min(2,1+2\beta,2+\beta)$. Obviously $p>1+\beta$ and the proof is concluded.

\end{enumerate}
\end{proof}

We leverage the previous lemma, which enables us to study the law of a single $M_k^t[g]$, to joint distribution. By \eqref{e:decomposition} this will be enough to study the convergence of $\ddp{X_t}{g}$. Recall also $m_k[g]$ given by \eqref{eq:mkg}.

\begin{lemma}\label{l:jointConvergence}
  Let $\Theta>0$ and assume that $\lambda\beta\leq \kappa(g)\mu(1+\beta)$. There exists $C,\delta>0$ such that for any $n\in\cbr{0,1,\ldots,t}$ and  $(\theta _0,\theta_1,\ldots,\theta_n)\in \R^{n+1}$ satisfying $0\leq \theta_i\leq \Theta$  we have
$$\mbr{\E\kbr{\exp\rbr{i\suma{k}{0}{n}\rbr{\theta_k \frac{M_k^t[g]}{\rbr{e^{\lambda k} |X_{t-k-1}|}^{\pb}}}}}-\prd{k}{0}{n}\exp\rbr{\theta_k^{1+\beta}e^{\lambda}m_k[g]}}\leq Ce^{-\delta (t-n)}.$$
\end{lemma}
We stress that in the above lemma $n$ may depend on $t$ (which will be used in subsequent proofs). When it is fixed it implies
\newcommand{\ftk}{F_{t,k}}
\begin{cor}Let $n\n$ and assume that $\lambda\beta\leq \kappa(g)\mu(1+\beta)$. Then

$$\rbr{\frac{M_0^t[g]}{F_{t,0}},\frac{M_1^t[g]}{F_{t,1}},\ldots,\frac{M_n^t[g]}{F_{t,n}}}
\rightarrow^d(\zeta_{0},\zeta_{1},\ldots,\zeta_{n}),\quad \text{as }t\rightarrow+\infty.
$$

where $\cbr{\zeta_k}_{k\in{0,\ldots,n}}$ are independent $(1+\beta)$-stable random variables with the characteristic functions
$$\E e^{i\theta\zeta_{k}}=e^{\theta^{1+\beta} e^{\lambda}m_k[g]},$$
where $\ftk := \rbr{e^{\lambda k} |X_{t-k-1}|}^{\pb}$.
\end{cor}

\begin{proof}
Let us denote $\gamma_{t,k} := M_k^t[g]/\ftk={\Delta_{t-k}[g_{k}]}/{\ftk}$ and a conditional characteristic function
\begin{equation}\label{e:varphitk}
  \varphi_{t,k}(\theta) := \E\kbr{\exp\rbr{i\theta\gamma_{t,k}}|\mathcal{F}_{t-k-1}},\quad 0\leq \theta \leq \Theta.
\end{equation}
Let $\varepsilon>0$ to be fixed later. Recall the event $\mathcal{E}_{t}(\varepsilon,M)$ defined in \eqref{e:et_event}, by Fact~\ref{f:maximumestimation} we take $M>0$ so that $\PR(\mathcal{E}_{t}(\varepsilon,M))\geq 1 - Ce^{-p_1 t}$ holds for some $C,p_1>0$ and further shortcut $\mathcal{E}_{t}=\mathcal{E}_{t}(\varepsilon,M)$. We define a modified conditional characteristic function
\begin{equation}\label{e:modofied}
\hat{\varphi}_{t,k}(\theta)
:=\E\kbr{\exp\rbr{i\theta\gamma_{t,k}}|\mathcal{F}_{t-k-1}} \indyk_{\mathcal{E}_{t-k-1}},\quad 0\leq \theta \leq \Theta.
\end{equation}
Using the inequality $|e^{ix}|\leq 1$ we estimate
\begin{equation}\label{e:epscondition}
\E\mbr{\varphi_{t,k}(\theta)-\hat{\varphi}_{t,k}(\theta)}\leq Ce^{-p_1 (t-k)}.
\end{equation}

Now we will study $\hat{\varphi}_{t,k}$. Recall that conditionally on $\mathcal{F}_{t-k-1}$, subsystems $X_{1}^{u,t-k-1}$
are independent, thus
$$\hat{\varphi}_{t,k}(\theta)=
\prd{u}{1}{|X_{t-k-1}|}\E\kbr{\exp\cbr{i\frac{\theta}{\ftk}
\rbr{\ddp{X_{1}^{u,t-k-1}}{g_{k}}-\PG_{1}^{\lambda}g_{k}(X_{t-k-1}(u))}}|\mathcal{F}_{t-k-1}} \indyk_{\mathcal{E}_{t-k-1}}.$$
We use Lemma~\ref{l:singlesubsystem} with $x = X_{t-k-1}(u)$ obtaining
\begin{equation}
  \label{e:product}
  \hat{\varphi}_{t,k}(\theta)=\prd{u}{1}{| X_{t-k-1}|}\kbr{1+Z_{g_{k}}\rbr{X_{t-k-1}(u);\frac{\theta}{\ftk}}
  +err\rbr{X_{t-k-1}(u);\frac{\theta}{\ftk}}}\indyk_{\mathcal{E}_{t-k-1}}\\
\end{equation}

Our aim now is showing that the total contribution of $err$ terms is small. By Lemma~\ref{l:singlesubsystem} on the event $\mathcal{E}_{t-k-1}$ for some $R\in \pspace$ we have
\begin{align*}
	err\rbr{X_{t-k-1}(u);\frac{\theta}{\ftk}} &\leq \rbr{\frac{e^{(\lambda-\kappa(g)\mu)k}}{\rbr{e^{\lambda k} |X_{t-k-1}|}^{\pb}}}^p 		R(X_{t-k-1}(u))\\
	&\leq \rbr{\frac{e^{(\lambda-\kappa(g)\mu)k}}{\rbr{e^{\lambda k}e^{(\lambda-\varepsilon)(t-k-1)}}^{\pb}}}^p Q_1(t-k-1),
\end{align*}
where $Q_1$ is some polynomial. Now using $\lambda\beta\leq \kappa(g)\mu(1+\beta)$ we estimate
\begin{align*}
  err\rbr{X_{t-k-1}(u);\frac{\theta}{\ftk}} &\leq \rbr{\frac{e^{(\lambda-\lambda \beta/(1+\beta))k}}{\rbr{e^{\lambda k}e^{(\lambda-\varepsilon)(t-k-1)}}^{\pb}}}^p Q_1(t-k-1) \\
  &= e^{\frac{p\varepsilon (t-k-1)}{1+\beta}}e^{-{\frac{p\lambda(t-k-1)}{1+\beta}}} Q_1(t-k-1).
\end{align*}
Since on $\mathcal{E}_{t-k-1}$ we have $|X_{t-k-1}|\leq e^{(\lambda+\varepsilon)(t-k-1)}$, we obtain
\begin{align*}
  \sum_{u=1}^{|X_{t-k-1}|}err\rbr{X_{t-k-1}(u);\frac{\theta}{\ftk}}\indyk_{\mathcal{E}_{t-k-1}}&\leq e^{(\lambda+\varepsilon)(t-k-1)} e^{\frac{p\varepsilon (t-k-1)}{1+\beta}}e^{-{\frac{p\lambda(t-k-1)}{1+\beta}}} Q_1(t-k-1).
\end{align*}
If $\varepsilon$ were $0$ the exponent in the expression above would be $\rbr{\frac{p}{1+\beta}-1}\lambda>0$ as by Lemma \ref{l:singlesubsystem} $p>1+\beta$. Thus we now fix $\varepsilon>0$ so that for some $C,\delta_1>0$
\begin{align}\label{e:error}
  \sum_{u=1}^{|X_{t-k-1}|}err\rbr{X_{t-k-1}(u);\frac{\theta}{\ftk}}\indyk_{\mathcal{E}_{t-k-1}}\leq C e^{-\delta_1(t-k)}.
\end{align}

Having dealt with the error terms, we analyse the contribution of $Z_{g_k}$ in \eqref{e:product}. We define
\begin{align*}
L_{t,k} := & \suma{u}{1}{|X_{t-k-1}|} Z_{g_{k}}\rbr{X_{t-k-1}(u);\frac{\theta}{\ftk}} =  \theta^{1+\beta}e^{-\lambda k}\frac{1}{|X_{t-k-1}|}\suma{u}{1}{|X_{t-k-1}|}Z_{g_{k}}\rbr{X_{t-k-1}(u)}\\
=&\theta^{1+\beta}e^{-\lambda k}\frac{\ddp{X_{t-k-1}}{Z_{g_k}}}{|X_{t-k-1}|},
\end{align*}
where we used $\ftk = \rbr{e^{\lambda k} |X_{t-k-1}|}^{\pb}$ and \eqref{e:zhtheta}. Recalling the definition of $m_k[g]$ in \eqref{eq:mkg} and using Lemma~\ref{l:lln} we obtain
\begin{equation}
\E\mbr{\rbr{L_{t,k}-\theta^{1+\beta}e^{\lambda}m_k[g] }\indyk_{\mathcal{E}_{t-k-1}}}\leq C_1e^{-\lambda k}e^{-p_1(t-k)}e^{(\lambda-\kappa(g)\mu)(1+\beta)k}\leq C_1e^{-p_1(t-k)}, \nonumber
\end{equation}
for some $C_1, p_1>0$, where in the last estimation we used $\lambda\beta\leq \kappa(g)\mu(1+\beta)$.
Hence using Fact~\ref{f:maximumestimation} we have also that
\begin{equation}\label{e:ltk}
\E\mbr{\rbr{L_{t,k}\indyk_{\mathcal{E}_{t-k-1}}-\theta^{1+\beta}e^{\lambda}m_k[g]}}\leq C_2 e^{-p_2(t-k)},
\end{equation}
for some $C_2,p_2>0$. We define
\begin{equation}\label{e:atk}
\mathcal{A}_{t,k} :=  \cbr{\mbr{L_{t,k}\indyk_{\mathcal{E}_{t-k-1}}-\theta^{1+\beta}e^{\lambda}m_k[g]}\leq C_2e^{-\frac{p_2}{2}(t-k)}}.
\end{equation}
Applying Markov's inequality with \eqref{e:ltk} we get that $\PR(\mathcal{A}_{t,k})\geq 1-e^{-\frac{p_2}{2}(t-k)}$.

Now we want to apply Fact~\ref{f:exp} to \eqref{e:product} with $n=|X_{t-k-1}|$,  $a_i^n = Z_{g_k}(\ldots) + err(\ldots)$ and $a=\theta^{1+\beta}e^{\lambda}m_k[g]$, which is by Lemma~\ref{l:zgmk} bounded by some constant $C_a$. Using \eqref{e:error} and \eqref{e:atk} we bound $b\leq  C_3 e^{-\delta_3 (t-k)}$ for some $C_3, \delta_3>0$. To check the assumption we use the fact that on $\mathcal{E}_{t-k-1}$ we have $\underset{{u\leq |X_{t-k-1}|}}{\max}|X_{t-k-1}(u)|\leq M (t-k-1)$ and thus using Lemma~\ref{l:zgmk} we get
$$Z_{g_{k}}\rbr{X_{t-k-1}(u);\frac{\theta}{\ftk}}\leq \frac{\theta^{1+\beta}}{|X_{t-k-1}|}R(X_{t-k-1}(u))\leq \frac{Q(t-k-1)}{|X_{t-k-1}|}$$ for some polynomial $Q$. Using the condition $|X_{t-k-1}|\geq e^{(\lambda-\varepsilon)(t-k-1)}$ from  $\mathcal{E}_{t-k-1}$ we get $|a^n_i|\leq n^{-3/4}$, when $t-k$ is sufficiently large. Finally,  by Fact~\ref{f:exp}
\begin{equation*}
\mbr{\rbr{\hat{\varphi}_{t,k}(\theta)-\exp\rbr{\theta^{1+\beta}e^{\lambda}m_k[g]}} \indyk_{\mathcal{A}_{t,k}}}\leq e^{2C_a} \rbr{\frac{C}{n^{1/2}}+C_3 e^{-\delta_3 (t-k)}} \leq C_4 e^{-\delta_4 (t-k)}
\end{equation*}
for some $C_4, \delta_4>0$. In the second step we used that on $\mathcal{E}_{t-k-1}$ we have $|X_{t-k-1}|\geq e^{(\lambda - \varepsilon)(t-k-1)}$. Recalling \eqref{e:epscondition} and dealing with $\mathcal{A}_{t,k}$ in the same way we obtain that for some $C_5, \delta_5>0$ there is
\begin{equation}\label{e:convergence}
\mbr{ {\varphi}_{t,k}(\theta)-\exp\rbr{\theta^{1+\beta}e^{\lambda}m_k[g]}}\leq C_5 e^{-\delta_5 (t-k)}.
\end{equation}

The above bound is strong enough to obtain the thesis of the lemma. Let us define for $\ell \in \cbr{-1, 0,\ldots,n}$
$$\varphi_{t}^\ell(\theta_{0},\ldots,\theta_{n}) :=
\prd{k}{0}{\ell}\exp\rbr{\theta_k^{1+\beta}e^{\lambda}m_k[g]} \E\exp\cbr{i\rbr{\sum_{k=\ell+1}^n \theta_{k}\gamma_{t,k}}}, \quad 0\leq \theta_i\leq \Theta,$$
where by convention the product is $1$ for $\ell=-1$. Conditioning with $\mathcal{F}_{t-\ell}$ we get
\begin{multline*}
  \varphi_{t}^\ell(\theta_{0},\ldots,\theta_{n}) - \varphi_{t}^{\ell+1}(\theta_{0},\ldots,\theta_{n}) = \prd{k}{0}{\ell}\exp\rbr{\theta_k^{1+\beta}e^{\lambda}m_k[g]}\\ \E\cbr{\rbr{\E[e^{i \theta_{\ell+1} \gamma_{t,\ell+1}}|\mathcal{F}_{t-\ell}] - e^{\theta_{\ell+1}^{1+\beta}e^{\lambda}m_{\ell+1}[g]}}\exp\cbr{i\rbr{\sum_{k=\ell+2}^n \theta_{k}\gamma_{t,k}}}}.
\end{multline*}
Using the elementary inequality $|e^{iz}|\leq 1$ and the fact that $\mbr{\exp\rbr{\theta_k^{1+\beta}e^{\lambda}m_k[g]}}\leq 1$ for any $k$ (note that it is a characteristic function) we obtain
\begin{equation*}
  |\varphi_{t}^\ell(\theta_{0},\ldots,\theta_{n}) - \varphi_{t}^{\ell+1}(\theta_{0},\ldots,\theta_{n})| \leq  \E{\mbr{\E[e^{i \theta_{\ell+1} \gamma_{t,\ell+1}}|\mathcal{F}_{t-\ell}] - e^{\theta_{\ell+1}^{1+\beta}e^{\lambda}m_{\ell+1}[g]}}}.
\end{equation*}
Recalling \eqref{e:varphitk} and applying~\eqref{e:convergence} we get
$$|\varphi_{t}^\ell(\theta_{0},\ldots,\theta_{n}) - \varphi_{t}^{\ell+1}(\theta_{0},\ldots,\theta_{n})|\leq C_5 e^{-\delta_5 (t-\ell-1)}.$$
To finish the proof we observe that the left hand side of thesis of the lemma is equal to $|\varphi_{t}^{(-1)}(\theta_{0},\ldots,\theta_{n}) - \varphi_{t}^{n}(\theta_{0},\ldots,\theta_{n})|$.
\end{proof}

\section{Large branching rate}\label{s:large}

Let us fix $g\in \pspace$ satisfying $\kappa(g)\geq 1$, in this section we assume that parameters of the system fulfill
\begin{equation*}
	\lambda>\kappa(g)\mu\rbr{1+\frac{1}{\beta}}.
\end{equation*}

For $p=(p_1,\ldots,p_d) \in\Z^d_+$ we recall the normalized Hermite polynomial $h_p$ given by \eqref{e:normalizedHermite} and define a~process $\{H_t^p\}_{t\geq 0}$ by
$$H_t^p := e^{-(\lambda-|p|\mu)t}\ddp{X_t}{h_p}=e^{-(\lambda-|p|\mu)t}\suma{u}{1}{|X_t|} h_p(X_t(i))$$
\begin{lemma} \label{l:martingale}
$H^p$ is a martingale with respect to $\mathcal{F}_t$. Moreover for any $0\leq \gamma< \beta$ we have $\underset{t\geq 0}{\sup} \nbr{H_t^p}_{L^{1+\gamma}}<\infty$ and therefore there exists
$$H_{\infty}^p := \gr{t}{\infty} H_t^p,$$
where the convergence is almost sure and in $L^{1+\gamma}$.
\end{lemma}
\begin{proof} We prove that $H_t^p$ is a martingale in Lemma \ref{l:generalmartingale} for a more general class of processes.

We assume that $\gamma<\beta$ is close enough to $\beta$ so that $\lambda>\kappa(g)\rbr{1+\frac{1}{\gamma}}\mu$ holds. Obviously, by the definition of $\kappa$ in \eqref{e:kappaf} we have $\kappa(h_p)=|p|$ and thus by Lemma~\ref{l:crucial} we obtain
$$\nbr{H_t^p}_{L^{1+\gamma}}\leq Ce^{-(\lambda-|p|\mu)t}e^{(\lambda-|p|\mu)t}=C,$$
for some $C>0$.  Hence the martingale is bounded in $L^{1+\gamma}$ and thus is convergent in $L^{1+\gamma}$ and almost surely.
\end{proof}

We are ready to formulate the main convergence theorem in the case of large branching rate.
\begin{thm}\label{t:supercritical} Let $g\in\mathcal{P}$ and assume that $\lambda>(1+1/\beta)\kappa(g)\mu$,  then
$$\frac{\ddp{X_t}{g}}{e^{(\lambda-\kappa(g)\mu) t}}\rightarrow \sum\limits_{|p|=\kappa(g)} \ddp{g}{h_p}_{\varphi} H_{\infty}^p,\quad \text{as }t\to +\infty.$$
The convergence holds in $L^{1+\gamma}$ for any $0\leq \gamma< \beta$. Moreover, if $g$ is twice differentiable and $|D^2 g| \in \pspace$ then the convergence is also almost sure.  
\end{thm}
The proof is presented in the two subsequent sections for the $L^{1+\gamma}$ and a.s. convergence respectively. Before the proof we present a corollary in the most common case of the second order analysis
\begin{cor}\label{c:supercritical_simple}Let $f\in\mathcal{P}$ and assume that $\lambda>(1+1/\beta)\mu$,  then
$$\frac{\ddp{X_t}{f} - |X_t|\ddp{\varphi}{f}}{e^{(\lambda-\mu) t}}\rightarrow \sum\limits_{|p|=1} \ddp{f}{h_p}_{\varphi} H_{\infty}^p,\quad \text{as }t\to +\infty.$$
The convergence holds in $L^{1+\gamma}$ for any $0\leq \gamma< \beta$. Moreover, if $f$ is twice differentiable and $|D^2 f| \in \pspace$ then the convergence is also almost sure.  
\end{cor}

\subsection{Convergence in $L^{1+\gamma}$}
We recall \eqref{e:l2scalarproduct} and decompose
\begin{equation}\label{el:division}
e^{-(\lambda-\kappa(g)\mu)t}\ddp{X_t}{g}=e^{-(\lambda-\kappa(g)\mu)t}\displaystyle\sum\limits_{|p|=\kappa(g)}\ddp{g}{h_p}_{\varphi} \ddp{X_t}{h_p}+e^{-(\lambda-\kappa(g)\mu)t}\ddp{X_t}{\tilde{g}},
\end{equation}
where
\begin{equation}\label{el:tildeg}
\tilde{g} := g-\displaystyle\sum\limits_{|p|=\kappa(g)}\ddp{g}{h_p}_{\varphi} h_p.
\end{equation}
We assume that $\gamma$ is so close to $\beta$ that $\kappa(g)(1+\gamma)\mu-\lambda\gamma<0$ holds. We observe that $\kappa(\tilde g)\geq  \kappa(g)+1$ and applying Lemma~\ref{l:crucial} to $\tilde{g}$ we get for some $C>0$ 
$$||e^{-(\lambda-\kappa(g)\mu)t}\ddp{X_t}{\tilde{g}}||_{L^{1+\gamma}}\leq C
\begin{cases}
e^{-(\lambda-\kappa(g)\mu)t}e^{(\lambda-(\kappa(g)+1)\mu)t} & \hbox{  if  }\gamma\lambda> (\kappa(g)+1)(1+\gamma)\mu\\
te^{-(\lambda-\kappa(g)\mu)t}e^{\frac{\lambda}{1+\gamma}t} & \hbox{  if  }\gamma\lambda=(\kappa(g)+1)(1+\gamma)\mu\\
e^{-(\lambda-\kappa(g)\mu)t}e^{\frac{\lambda}{1+\gamma}t} & \hbox{  if  }\gamma\lambda<(\kappa(g)+1)(1+\gamma)\mu\\
\end{cases}$$
In the first case the exponent near $t$ is equal to
$\rbr{\kappa(g)\mu-\lambda+\lambda-(\kappa(g)+1)\mu}t=-\mu t.$
In the second and the third case the exponent is equal to\\
$\rbr{\kappa(g)\mu-\lambda+\frac{\lambda}{1+\gamma}}t
=\rbr{\frac{1}{1+\gamma}\rbr{\kappa(g)(1+\gamma)\mu-\lambda\gamma}}t.$
Hence in all cases we have
\begin{equation}\label{el:remainder}
\nbr{e^{-(\lambda-\kappa(g)\mu)t}\ddp{X_t}{\tilde{g}}}_{L^{1+\gamma}}\leq e^{-\delta t}
\end{equation}
for some $C, \delta>0$. The decomposition~\eqref{el:division} combined with~\eqref{el:remainder} and Lemma~\ref{l:martingale} shows that the convergence in $L^{1+\gamma}$ holds.

\subsection{Almost sure convergence}
Let $a \in \R$ and $h\in\mathcal{P}$ be such that $Lh \in \pspace$. We set $\bar{h}=Lh+a\mu h$, where $L$ is infinitesimal operator \eqref{eq:infinitesimal_operator}. We define a process $\cbr{M_t^{h,a}}_{t\geq 0}$ by
$$M^{h,a}_t := e^{-(\lambda-a\mu)t}\ddp{X_t}{h}-\int_{0}^t e^{-(\lambda-a\mu)w}\ddp{X_w}{\bar{h}}\text{d}w.$$
We notice that when $h=h_p$ and $a=|p|$ then $\bar{h}=0$ and $M^{h,a} = H^p$ defined in the previous section. Importantly, we have
\begin{lemma}\label{l:generalmartingale}
Let $a \in \R$ and $h\in\mathcal{P}$ be such that $Lh \in \pspace$. Then the process $M_t^{h,a}$ is a martingale.
\end{lemma}
\begin{proof}Using many-to-one formula \eqref{e:manytoone} we compute
\begin{equation}\label{el:MtFs}
\E\rbr{\ddp{X_t}{h}|\mathcal{F}_s}=\E\rbr{\suma{u}{1}{|X_s|} \ddp{X_s}{h}|\mathcal{F}_s}=
\suma{u}{1}{|X_s|} \PG_{t-s}^{\lambda}h(X_u)=\ddp{X_s}{\PG_{t-s}^{\lambda}h}.
\end{equation}
Further we get
\begin{align}\label{el:integralpart}
\E\kbr{\rbr{\calka{0}{t} e^{-(\lambda-a\mu)w}\ddp{X_t}{\bar{h}}\text{d}w}|\mathcal{F}_s}&=
\calka{0}{s} e^{-(\lambda-a\mu)w}\ddp{X_w}{\bar{h}} \text{d}w\nonumber\\
&+\E\kbr{\rbr{\calka{s}{t} e^{-(\lambda-a\mu)w}\ddp{X_w}{\bar{h}} \text{d}w}|\mathcal{F}_s}
\nonumber\\
&=\calka{0}{s} e^{-(\lambda-a\mu)w}\ddp{X_w}{\bar{h}} \text{d}w\nonumber\\
&+\calka{s}{t} e^{-(\lambda-a\mu)w}\ddp{X_s}{\PG_{w-s}^{\lambda}\bar{h}} \text{d}w .
\end{align}
Recall \eqref{e:pg}, to identify the second integral we compute the derivative
\begin{align*}
\frac{\text{d}}{\text{d}w}\rbr{e^{-(\lambda-a\mu)w}\ddp{X_s}{\PG_{w-s}^{\lambda}h}}&=
\frac{\text{d}}{\text{d}w}\rbr{e^{-\lambda s}e^{a\mu w}\ddp{X_s}{\PG_{w-s}h}}\nonumber\\
&=a\mu e^{-\lambda s}e^{a\mu w}\ddp{X_s}{\PG_{w-s}h}
+ e^{-\lambda s}e^{a\mu w}\ddp{X_s}{L\PG_{w-s}h}\nonumber\\
&=e^{-(\lambda-a\mu)w}\ddp{X_s}{\PG_{w-s}^{\lambda} (a\mu h+Lh)} = e^{-(\lambda-a\mu)w}\ddp{X_s}{\PG_{w-s}^{\lambda}\bar{h}}.
\end{align*}
Hence
\begin{equation}\label{el:integration}
\calka{s}{t} e^{-(\lambda-a\mu)w}\ddp{X_s}{\PG_{w-s}^{\lambda}\bar{h}} \text{d}w=
e^{-(\lambda-a\mu)t}\ddp{X_t}{\PG_{t-s}^{\lambda}h}
-e^{-(\lambda-a\mu)s}\ddp{X_s}{h}.
\end{equation}
In consequence combining~\eqref{el:MtFs}, \eqref{el:integralpart} and \eqref{el:integration} we obtain that $\E\kbr{M^{h,a}_t|\mathcal{F}_s}=M^{h,a}_s.$
\end{proof}
We denote
$$L^{h,a}_t := \calka{0}{t} e^{-(\lambda-a\mu)w}\ddp{X_w}{\bar{h}}\text{d}w.$$

In order to prove the almost sure convergence in Theorem~\ref{t:supercritical} we need to show\\
$e^{-(\lambda-\kappa(g)\mu)t}\ddp{X_t}{\tilde{g}}\to 0$ almost surely, where $\tilde{g}$ is given by~\eqref{el:tildeg} and satisfies $\kappa(\tilde{g})\geq  \kappa(g)+1$. We take $a=\kappa(g)+\frac{1}{2}$ and obtain
\begin{equation}\label{el:M+L}
e^{-(\lambda-\kappa(g)\mu)t}\ddp{X_t}{\tilde{g}}=e^{-\frac{1}{2}\mu t}(M^{\tilde{g},a}_t
+L^{\tilde{g},a}_t).
\end{equation}

We will show the almost sure convergence of both parts
\begin{lemma}\label{l:emtMt}
For $a=\kappa(g)+\frac{1}{2}$ we have
$$e^{-\frac{1}{2}\mu t}M_t^{\tilde{g},a}\rightarrow 0,\quad e^{-\frac{1}{2}\mu t}L_t^{\tilde{g},a}\rightarrow 0, \quad \text{as }t\rightarrow +\infty.$$
The convergences hold almost surely.
\end{lemma}
\begin{proof}
Since $L$ preserves its eigenspaces, we have $\kappa (L \tilde{g}+a\mu\tilde{g})=\kappa(\tilde{g})\geq  \kappa(g)+1$. The same way as in the proof of \eqref{el:remainder} we obtain that
\begin{equation}\label{e:est}
\nbr{e^{-\rbr{\lambda-a\mu}t}\ddp{X_t}{\tilde{g}}}_{L^{1+\gamma}}\leq C_1 e^{-\delta t},\quad
\nbr{e^{-\rbr{\lambda-a\mu}t}\ddp{X_t}{L \tilde{g}+a\tilde{g}}}_{L^{1+\gamma}}\leq C_1 e^{-\delta t},
\end{equation}
for some $C_1, \delta>0$. We denote $ Y_t=\calka{0}{t} e^{-(\lambda-a\mu)w}\mbr{\ddp{X_w}{L \tilde{g}+a\mu\tilde{g}}}dw$. Clearly $|L^{\tilde g, a}_t|\leq Y_t$. Using the previous estimation and the triangle inequality we obtain that
\begin{equation}\label{e:intest}
\nbr{L^{\tilde g, a}_t}_{L^{1+\gamma}}\leq\nbr{Y_t}_{L^{1+\gamma}}\leq C_1 \calka{0}{t} e^{-\delta w}dw\leq \frac{C_1}{\delta}.
\end{equation}
As $Y_t$ is increasing we conclude that it converges to some $Y_\infty$ in $L^{1+\gamma}$ and almost surely. This directly implies the second convergence. To get the first one we observe that by~\eqref{e:est} and~\eqref{e:intest} the martingale $M^{\tilde{g},a}_t$ is bounded in $L^{1+\gamma}$ and hence it converges almost surely.
\end{proof}

Using Lemma~\ref{l:emtMt} and~\eqref{el:M+L} we get that
$$e^{-(\lambda-\kappa(g)\mu)t}\ddp{X_t}{\tilde{g}}\rightarrow 0,$$
almost surely. Combining this result with Lemma~\ref{l:martingale} we obtain the almost sure convergence in Theorem~\ref{t:supercritical}.

\begin{rem}
This proof holds also for the case of $\beta=1$ i.e finite-variance case. It is even simpler, since instead of estimations for moments we have equalities and we do not need to take $\gamma<\beta$.
\end{rem}

\section{Critical branching rate}\label{s:critical}
Let us fix $g\in \pspace$ satisfying $\kappa(g)\geq 1$. In this section we assume that parameters of the system fulfill
\begin{equation}\label{ec:critical}
\lambda=\kappa(g)\mu\rbr{1+\frac{1}{\beta}} \iff \lambda=(\lambda-\kappa(g)\mu)(1+\beta)\iff \lambda\rbr{1-\frac{1}{1+\beta}}=\kappa(g)\mu.
\end{equation}
Recall definitions \eqref{e:zhtheta} and \eqref{eq:mkg}. First we will identify a parameter of the limiting distribution. Aiming at this goal we will show a more convenient way to write $m_k[g]$.
\begin{lemma}\label{l:mkalternative} We have
\begin{equation}\label{eq:mkalternative}
m_k[g]=\lambda \calka{\R^d}{} \rbr{\calka{k}{k+1} e^{-\lambda s}\kbr{\rbr{\PG^{\lambda}_{s} (ig)}^{1+\beta}}(x)  \text{d}s} \varphi(x) \text{d}x.
\end{equation}
\end{lemma}
\begin{proof}
Using Fubini's theorem and the fact that $\varphi$ is an invariant measure of $\PG$ we compute
\begin{align*}
m_k[g]&=\lambda e^{-\lambda (k+1)}\calka{0}{1}\calka{\R^d}{} \PG_{1-s}^{\lambda}\kbr{\rbr{\PG^{\lambda}_{s+k} (ig)}^{1+\beta}}(x)  \varphi(x) \text{d}x \text{d}s \nonumber\\
&=\lambda  \calka{\R^d}{} \rbr{\calka{0}{1} e^{-\lambda (k+s)}\kbr{\rbr{\PG^{\lambda}_{s+k} (ig)}^{1+\beta}}(x)  \text{d}s} \varphi(x) \text{d}x\nonumber.
\end{align*}
Now the proof follows using substitution $s=s+k$.
\end{proof}
Now we describe the limit in the critical case. Let $h$ be the projection of $g$ onto the $\kappa(g)$-eigenspace i.e.
\begin{equation}\label{e:h(g)}
h(x):=\displaystyle\sum\limits_{|p|=\kappa(g)}\ddp{g}{h_p}_{\varphi} h_p(x).
\end{equation}
\begin{lemma}\label{l:meanvalue}
Let $g\in\mathcal{P}$ and $\lambda\beta=\kappa(g)\mu(1+\beta)$. Then the following limit exists
\begin{equation}\label{ec:limit_mkg}
  \bar{m}[g] := \gr{t}{\infty} \frac{1}{t}\rbr{\suma{k}{0}{t} m_k[g]} = \lambda \ddp{(ih)^{1+\beta}}{\varphi}.
\end{equation}
\end{lemma}

\begin{proof}

By definition of $\kappa$ given in \eqref{e:kappaf} we have $h \not \equiv 0$ and, moreover, $h$ is an eigenfunction of $\PG_t$ with eigenvalue $e^{-\kappa(g)\mu t}$.  We observe further that $\kappa(h)=\kappa(g)$ and $\kappa(g-h)\geq  \kappa(g)+1$. Fact~\ref{f:semigroup} yields that for some $R_1\in\pspace$ we have
$$|\PG_s^{\lambda}g|\leq e^{(\lambda-\kappa(g)\mu)s}R_1(x), \quad |\PG_s^{\lambda}h|\leq e^{(\lambda-\kappa(g)\mu)s}R_1(x),\quad |\PG_s^{\lambda}(g-h)|\leq e^{(\lambda-(\kappa(g)+1)\mu)s}R_1(x).$$
We aim to show that $m_k[g] - m_k[h] = e^{-\lambda k}\ddp{Z_{g_k} - Z_{h_k}}{\varphi}$ are small. Using \eqref{e:1+beta} we obtain for some $R\in\pspace$
\begin{align*}
\mbr{\rbr{\PG^{\lambda}_{s} (ig_k)}^{1+\beta}(x)  -\rbr{\PG^{\lambda}_{s} (ih_k)}^{1+\beta}(x)  } &\leq \mbr{\rbr{i\PG^{\lambda}_{s+k} g}^{1+\beta}(x)  -\rbr{i\PG^{\lambda}_{s+k} h}^{1+\beta}(x)  } \\
& \leq e^{(\lambda-\kappa(g)\mu)\beta (s+k)}e^{(\lambda-(\kappa(g)+1)\mu)(s+k)}R(x)\\
&=e^{(\lambda-\mu) (s+k)}R(x),
\end{align*}
for some $R \in \pspace$. In the last equality we used~\eqref{ec:critical}. In consequence $\mbr{Z_{g_k}(x)-Z_{h_k} (x) }\leq e^{(\lambda-\mu)k}R_1(x)$ and further $|m_k[g] - m_k[h]|\leq Ce^{-\mu k}$ for some $C>0, R_1\in\pspace$. Hence, the limits of $(1/t)\sum_{k=0}^{t}m_k[g]$ and $(1/t)\sum_{k=0}^{t}m_k[h]$ coincide.
Using $\PG^{\lambda}_{s} h = e^{(\lambda - \kappa(g)\mu)s} h$, Lemma~\ref{l:mkalternative} and~\eqref{ec:critical} we obtain
\begin{equation*}
  m_k[h] = \lambda \calka{\R^d}{} \rbr{\calka{k}{k+1} e^{-\lambda s}\kbr{e^{(\lambda-\kappa(g)\mu)s}ih(x)}^{1+\beta} \text{d}s} \varphi(x) \text{d}x= \lambda \calka{\R^d}{}  (ih(x))^{1+\beta} \varphi(x) \text{d}x.
\end{equation*}
Hence
$\gr{t}{\infty} \frac{1}{t}\rbr{\suma{k}{0}{t} m_k[g]}= \lambda \ddp{(ih(x))^{1+\beta}}{\varphi(x)}.$
\end{proof}
We are ready to present the main result of this section.

\begin{thm}\label{t:critical}
Let $g\in\mathcal{P}$  and assume that $\lambda=(1+1/\beta)\kappa(g)\mu$, then
\begin{equation}\label{ec:thm}
\frac{\ddp{X_t}{g}}{(t|X_t|)^{\pb }}\rightarrow^d \eta,\quad t\to +\infty,
\end{equation}
where $\eta$ is a $(1+\beta)$-stable random variable with the characteristic function given  by
\begin{equation}\label{e:characteristic_critical}
  \E \exp\rbr{i\theta\eta}=\exp\rbr{\theta^{1+\beta}\bar{m}[g]}, \quad \theta\geq 0.
\end{equation}
\end{thm}
Before the proof we present a corollary in the most common case of the second order analysis
\begin{cor}\label{c:critical_simple} Let $f\in\mathcal{P}$ and assume that $\lambda=(1+1/\beta)\mu$,  then
  $$\frac{\ddp{X_t}{f} - \ddp{\varphi}{f}}{(t|X_t|)^{\pb }}\rightarrow^d \eta,\quad t\to +\infty.$$
  Set $g = f - \ddp{f}{\varphi}$. When $\kappa(g) = 1$ then $\eta$ has characteristic function \eqref{e:characteristic_critical}, otherwise $\eta = 0$.
\end{cor}

\begin{proof} Recalling decomposition~\eqref{e:decomposition} we write
\begin{equation}\label{ec:cut}
(t|X_t|)^{-\pb } {\ddp{X_t}{g}}{}=(t|X_t|)^{-\pb }\suma{k}{0}{\lfloor t-\ln t\rfloor } M_k^t[g]+(t|X_t|)^{-\pb }\suma{k}{\lceil t-\ln t\rceil}{t} M_k^t[g] =: I_t + J_t.
\end{equation}
We will show that $I_t \to^d \eta$ and $J_t \to^d 0$, where $\eta$ is as prescribed in the theorem.\\
\textbf{Convergence of $I_t \to^d \eta$}. We define
\begin{equation}
  \tilde{I}_t := \suma{k}{0}{\lfloor t-\ln t\rfloor }\frac{M_k^t[g]}{\rbr{t{e^{\lambda (k+1)} |X_{t-k-1}|}}^{1+\beta}}.\nonumber
\end{equation}
We apply Lemma~\ref{l:jointConvergence} with $n=\lfloor t-\ln t \rfloor$ and $\theta_k=\frac{\theta}{(te^{\lambda})^{\pb}}$  obtaining
$$\mbr{\E e^{i\theta \tilde{I}_t}-\exp\rbr{\theta^{1+\beta} \frac{1}{t}\rbr{\suma{k}{0}{\lfloor t-\ln t\rfloor} m_k[g]}}}\leq e^{-p\ln t}=t^{-p}$$
for some $p>0$. Now the convergence $\tilde{I}_t \rightarrow^d \eta$ follows by Lemma \ref{l:meanvalue}.

Using elementary inequality $\mbr{\prod z_i-\prod z_i'}\leq \sum |z_i-z_i'|$ valid for $z_i, z_i'\in\mathbb{C}, |z_i|, |z_i'|\leq 1$ we obtain that
\begin{equation}\label{ec:decomp}
  |\E e^{i\theta \tilde{I}_t} - \E e^{i\theta I_t}| \leq \suma{k}{0}{\lfloor t-\ln t\rfloor } \E|Y_{t,k}|,
\end{equation}
where
\begin{equation}\label{ec:singlesummand}
Y_{t,k} := \exp\rbr{i\theta \rbr{\frac{M_k^t[g]}{\rbr{te^{\lambda (k+1)} |X_{t-k-1}|}^{\pb}}}}-\exp\rbr{i\theta\rbr{\frac{M_k^t[g]}{\rbr{t|X_t|}^{\pb}}}}.
\end{equation}
We will show that $\E |Y_{t,k}|$ is small. Recall the martingale $W_t = e^{-\lambda t}|X_t|$ and denote the event
\begin{equation}
  B_{t,k} := \cbr{|W_t-W_{t-k-1}|\leq e^{-\frac{\lambda}{4} (t-k-1)}, W_{t-k-1}\geq  e^{-\frac{\lambda}{32}(t-k-1)}}.\nonumber
\end{equation}
Using $|Y_{t,k}|\leq 2$ and a union bound we get
\begin{align}\label{ec:completion}
\E |Y_{t,k}|\indyk_{B_{t,k}^c} &\leq 2 \PR\rbr{|W_t-W_{t-k-1}|> e^{-\frac{\lambda}{4}(t-k-1)}} + 2 \PR\rbr{W_{t-k-1}< e^{-\frac{\lambda}{32}(t-k-1)}} \nonumber\\
&\leq Ce^{-p(t-k-1)}.
\end{align}
for some $C,p>0$. To estimate the second term we use Fact \ref{f:maximumestimation}. For the first term we observe that by Fact~\ref{f:martingaleconvergence} we have $\E |W_t-W_{t-k-1}|\leq Ce^{-\frac{\lambda}{2} (t-k-1)}$ and then use use Markov's inequality.

By the elementary inequality $|e^{iz_1 }-e^{iz_2}|\leq |z_1 - z_2|$ valid for $z_1, z_2 \in \R$ and straightforward calculations we get
\begin{align}\label{ec:indyk}
\E |Y_{t,k}|\indyk_{B_{t,k}}&\leq \theta \E\mbr{M_k^t[g]\rbr{\frac{1}{\rbr{te^{\lambda (k+1)} |X_{t-k-1}|}^{\pb}}-\frac{1}{\rbr{t|X_t|}^{\pb}}}}\indyk_{B_{t,k}}\\
&=\E|M_k^t[g]|\frac{1}{t^{\pb}}\mbr{\frac{|X_t|^{\pb}-\rbr{e^{\lambda (k+1)} |X_{t-k-1}|}^{\pb}}{|X_t|^{\pb}\rbr{e^{\lambda (k+1)} |X_{t-k-1}|}^{\pb}}}\indyk_{B_{t,k}}\nonumber\\
&=\theta \frac{1}{t^{\pb}}e^{-\frac{\lambda}{1+\beta} t} \E|M_k^t[g]|\frac{\mbr{W_t^{\pb}-W_{t-k-1}^{\pb}}}{W_t^{\pb}W_{t-k-1}^{\pb}}
\indyk_{B_{t,k}}.\nonumber
\end{align}
Using the fact that $\{|W_t|\geq  e^{-\frac{\lambda}{16}(t-k-1)}\}\subset B_{t,k}$ and the elementary inequality
$|x^{1/(1+\beta)}-y^{1/(1+\beta)}|\leq C\max\rbr{|x|^{-\frac{\beta}{1+\beta}},|y|^{-\frac{\beta}{1+\beta}}}|x-y|$ valid for $x,y\in \R_{+}$, on the event $B_{t,k}$ we have
$$W_{t-k-1}^{\pb}\geq  e^{-\frac{\lambda}{32(1+\beta)}(t-k-1)}\geq  e^{-\frac{\lambda}{32}(t-k-1)}, \quad
W_t^{\pb}\geq   e^{-\frac{\lambda}{16(1+\beta)}(t-k-1)} \geq  e^{-\frac{\lambda}{16}(t-k-1)},$$
\begin{align*}
 \mbr{W_t^{\pb}-W_{t-k-1}^{\pb}} &\leq C\max\rbr{W_t^{-\frac{\beta}{1+\beta}},W_{t-k-1}^{-\frac{\beta}{1+\beta}}}|W_t-W_{t-k-1}|\\
 &\leq C\max \rbr{e^{\frac{\lambda}{32}(t-k-1)},e^{\frac{\lambda}{16}(t-k-1)}}e^{-\frac{\lambda}{4}(t-k-1)}=Ce^{-\frac{3}{16}\lambda (t-k-1)}.
\end{align*}
Putting these estimates into~\eqref{ec:indyk} and using Lemma~\ref{l:singledelta} we obtain that for every $\gamma<\beta$
\begin{equation}\label{ec:milestone}
  \E |Y_{t,k}|\indyk_{B_{t,k}}\leq Ce^{-\frac{\lambda}{1+\beta} t}\E|M_k^t[g]| e^{-\frac{3}{32}\lambda (t-k-1)}
  \leq C({\gamma})e^{-\frac{\lambda}{1+\beta} t} e^{\frac{\lambda}{1+\gamma}t} e^{\frac{\gamma\lambda-\kappa(g)(1+\gamma)\mu}{1+\gamma}k}e^{-\frac{\lambda}{32}(t-k-1)}.
\end{equation}
Simple calculations using~\eqref{ec:critical} lead to $\frac{\lambda}{1+\gamma}-\frac{\lambda}{1+\beta}=-\frac{\gamma\lambda-\kappa(g)(1+\gamma)\mu}{1+\gamma}$. Thus
\begin{equation}
  \E |Y_{t,k}|\indyk_{B_{t,k}} \leq  C({\gamma}) e^{\rbr{\frac{\lambda}{1+\gamma}-\frac{\lambda}{1+\beta} }(t-k)} e^{-\frac{\lambda}{32}(t-k-1)}.\nonumber
\end{equation}
Now we take $\gamma$ so close to $\beta$ that $\frac{\lambda}{1+\gamma}-\frac{\lambda}{1+\beta}=\frac{\lambda}{64}$. In the end we obtain
\begin{equation*}
\E |Y_{t,k}|\indyk_{B_{t,k}}\leq Ce^{-\frac{\lambda}{64}(t-k-1)},
\end{equation*}
for some $C>0$. Now using~\eqref{ec:decomp} and \eqref{ec:completion} we get that
\begin{equation}\label{ec:milestone2}
 |\E e^{i\theta \tilde{I}_t} - \E e^{i\theta I_t}| \leq Ce^{-q \ln t}=Ct^{-q},
\end{equation}
for some $C,q>0$. Hence, recalling $\tilde{I}_t\rightarrow^d \eta$, we also have $I_t \to^d \eta$.
\\
\textbf{Convergence of $J_t \to^d 0$}. By Lemma~\ref{l:singledelta} we know that for any $\gamma<\beta$ we have
$$\nbr{M_k^t[g]}_{L^{1}}\leq \nbr{M_k^t[g]}_{L^{1+\gamma}}\leq C(\gamma)
e^{\frac{\lambda}{1+\gamma}t} e^{\frac{\gamma\lambda-\kappa(g)(1+\gamma)\mu}{1+\gamma}k}= C(\gamma)e^{\frac{\lambda}{1+\gamma}(t-k)}e^{\frac{\lambda}{1+\beta}k},$$
for some $C(\gamma)>0$.
In the last step we used~\eqref{ec:critical}. By the triangle inequality we obtain
$$\nbr{\suma{k}{\lceil t-\ln t\rceil}{t} M_k^t[g] }_{L^1}\leq C_\gamma e^{\frac{\lambda}{1+\gamma}\ln t}e^{\frac{\lambda}{1+\beta}(t-\ln t)},$$
for some $C_\gamma >0$.
We denote the event $A_t:=\cbr{|X_t|\geq  t^{-1/2}e^{\lambda t}}$. By Fact~\ref{f:magi} we know that $1-\PR(A_t)\leq t^{-p}$ for some $p>0$. Hence using elementary inequality $|e^{iz}-1|\leq |z|$ valid for $z \in \R$ we get
\begin{align*}
\mbr{1-\E\kbr{\exp\rbr{i\theta J_t}}}&\leq
\E\rbr{\mbr{1-\kbr{\exp\rbr{i\theta J_t}}}\indyk_{A_t}}+\E\rbr{\mbr{1-\kbr{\exp\rbr{i\theta J_t}}}\indyk_{A_t^c}}\\
&\leq \E|J_t|\indyk_{A_t}+2\PR(A_t^c) \leq C_\gamma e^{\frac{\lambda}{1+\gamma}\ln t}e^{\frac{\lambda}{1+\beta}(t-\ln t)}\frac{1}{(t^{1/2}e^{\lambda t})^{\pb}}+2t^{-p}\\
&\leq C_\gamma{
e^{\rbr{\frac{\lambda}{1+\gamma}-\frac{\lambda}{1+\beta}}\ln t}t^{-\frac{1}{2(1+\beta)}}+2t^{-p}} = C_\gamma{
t^{\rbr{\frac{\lambda}{1+\gamma}-\frac{\lambda}{1+\beta}}-\frac{1}{2(1+\beta)}}+2t^{-p}}.
\end{align*}
We take $\gamma$ so close to $\beta$ that the first exponent is negative, thus $\mbr{1-\E\kbr{\exp\rbr{i\theta J_t}}}\rightarrow  0$. Consequently,  $J_t \to^d 0$, which combined with~\eqref{ec:cut} and $I_t \to^d \eta$ yields thesis of the theorem.
\end{proof}

\section{Small branching rate}\label{s:small}
Let us fix $g\in \pspace$ satisfying $\kappa(g)\geq 1$. In this section we assume that parameters of the system fulfill
\begin{equation}\label{es:subcritical}
\lambda<\rbr{1+\frac{1}{\beta}}\kappa(g)\mu \iff \lambda\beta<\kappa(g)\mu(1+\beta).
\end{equation}
Recall definitions \eqref{e:zhtheta} and \eqref{eq:mkg}. First we will identify a parameter of the limiting distribution.
\begin{lemma}\label{l:series}
The series $\suma{k}{0}{\infty} m_k[g]$ is absolutely convergent and
$$m[g]:=\suma{k}{0}{\infty} m_k[g]=\lambda e^{\lambda}\calka{\R^d}{} \kbr{\calka{0}{\infty} e^{-\lambda s}\kbr{\rbr{\PG^{\lambda}_{s} (ig)}^{1+\beta}}(x)  \text{d}s} \varphi(x) \text{d}x.$$
\end{lemma}
\begin{proof}
By Lemma~\ref{l:zgmk} we have $|m_k[g]|\leq Ce^{(\lambda\beta-\kappa(g)\mu(1+\beta))k}$ for some $C>0$. Moreover, Fact~\ref{f:semigroup} yields $|e^{-\lambda s}\rbr{\PG^{\lambda}_{s} g}^{1+\beta}|\leq e^{(\lambda\beta-\kappa(g)\mu(1+\beta))s}R(x)$ for some $R\in\mathcal{P}$. By~\eqref{es:subcritical}, the exponents are negative hence both the series and integral are absolutely convergent. Now using~\eqref{eq:mkalternative} we get the equality from the thesis.
\end{proof}
We are ready to present the main result of this section.
\begin{thm}\label{t:subcritical}
Let $g\in\mathcal{P}$ and assume that $\lambda<(1+1/\beta)\kappa(g)\mu$. Then
\begin{equation}\label{es:thm}
\frac{\ddp{X_t}{g}}{|X_t|^{\pb }}\rightarrow^d \zeta,\quad t\to +\infty,
\end{equation}
where $\zeta$ is a $(1+\beta)$-stable random variable with the characteristic function given by
\begin{equation}\label{e:characteristic_subcritical}
  \E \exp\rbr{i\theta\zeta}=\exp\rbr{\theta^{1+\beta}m[g]}, \quad \theta\geq 0.
\end{equation}
\end{thm}
Before the proof we present a corollary in the most common case of the second order analysis
\begin{cor}\label{c:subcritical_simple} Let $f\in\mathcal{P}$ and assume that $\lambda<(1+1/\beta)\mu$,  then
  $$\frac{\ddp{X_t}{f} - \ddp{\varphi}{f}}{|X_t|^{\pb }}\rightarrow^d \zeta,\quad t\to +\infty.$$
  Set $g = f - \ddp{f}{\varphi}$. When $\kappa(g) = 1$ then $\zeta$ has characteristic function \eqref{e:characteristic_subcritical}, otherwise $\zeta = 0$.
\end{cor}

\begin{proof}
Recalling the decomposition~\eqref{e:decomposition} we write
\begin{equation}\label{es:cut}
|X_t|^{-\pb } {\ddp{X_t}{g}}{}=|X_t|^{-\pb }\suma{k}{0}{\lfloor t-\ln t\rfloor } M_k^t[g]+|X_t|^{-\pb }\suma{k}{\lceil t-\ln t\rceil}{t} M_k^t[g] =: I_t + J_t.
\end{equation}
We will show that $I_t \to^d \zeta$ and $J_t \to^d 0$, where $\zeta$ is as prescribed in the theorem.\\
\textbf{Convergence of $I_t \to^d \zeta$}.
We define
\begin{equation}
  \tilde{I}_t := \suma{k}{0}{\lfloor t-\ln t\rfloor }\frac{M_k^t[g]}{\rbr{{e^{\lambda (k+1)} |X_{t-k-1}|}}^{1+\beta}}\nonumber.
\end{equation}
Applying Lemma~\ref{l:jointConvergence} with $n=\lfloor t-\ln t \rfloor$ and $\theta_k=\frac{\theta}{(e^{\lambda})^{\pb}}$ yields
$$\mbr{\E e^{i\theta \tilde{I}_t}-\exp\rbr{\theta^{1+\beta} \rbr{\suma{k}{0}{\lfloor t-\ln t \rfloor} m_k[g]}}}\leq e^{-p\ln t}=t^{-p}$$
for some $p>0$. Now the convergence $\tilde{I}_t \rightarrow^d \zeta$ follows by Lemma~\ref{l:series}.

The proof of this part concludes once we show
\begin{equation*}
  |\E e^{i\theta \tilde{I}_t} - \E e^{i\theta I_t}| \to 0.
\end{equation*}
To this end one can follow the lines of the proof in the critical case. Starting from \eqref{ec:decomp} the calculations are virtually unchanged until \eqref{ec:milestone}. Further, we use the inequality $\frac{\lambda}{1+\gamma}-\frac{\lambda}{1+\beta}\leq-\frac{\gamma\lambda-\kappa(g)(1+\gamma)\mu}{1+\gamma}$ instead of equality which is even better and lets us to arrive to \eqref{ec:milestone2}.
\\
\textbf{Convergence of $J_t \to^d 0$}. By Lemma~\ref{l:singledelta} we know that for any $\gamma<\beta$ we have
$$\nbr{M_k^t[g]}_{L^{1}}\leq \nbr{M_k^t[g]}_{L^{1+\gamma}}\leq C(\gamma)
e^{\frac{\lambda}{1+\gamma}t} e^{\frac{\gamma\lambda-\kappa(g)(1+\gamma)\mu}{1+\gamma}k}\leq C(\gamma)e^{-\delta t}e^{\frac{\lambda}{1+\gamma}(t-k)}e^{\frac{\lambda}{1+\beta}k}$$
for some $C(\gamma),\delta>0$.
In the last step we used~\eqref{es:subcritical}. Proceeding in a very similar way as in the critical case we get
\begin{equation*}
\mbr{1-\E\kbr{\exp\rbr{i\theta J_t}}}\to 0.
\end{equation*}
Thus $J_t \to^d 0$, what combined with~\eqref{es:cut} and $I_t \to^d \zeta$ yields the thesis of the theorem.
\end{proof}

\section{Conclusions and open questions}\label{s:open}
In the paper we presented a comprehensive picture of higher order behaviour of a branching particle system whose branching law has heavy tails. This goes in par with finding of \cite{Adamczak2015} and \cite{Ren2014} for the system with finite variance branching law.

Guided by the findings of \cite{Ren2017c} we conjecture that the qualitative behaviour is similar for any system in which particles move according to strongly mixing diffusion process (though the case of the Ornstein-Uhlenbeck process is particularly easy due to its simple spectral decomposition).

We conjecture also that in the subcritical and supercritical cases the qualitative behaviour does not depend on the branching law. More precisely consider a branching law in the domain of attraction of $(1+\beta)$-stable law, with the generating function $F(s)=ms-(m-1)+(m-1)(1-s)^{1+\beta}L(1-s)$, where $L$ is a slowly varying function at $0$. We conjecture that the statements of Theorem \ref{t:supercritical} and Theorem \ref{t:subcritical} hold true with only minor modifications. The critical case seems to be more subtle, in our preliminary calculations we found that the behaviour may fall in either regime depending on the properties of $L$.

\subsection*{Acknowledgments} We would like to thank prof. A. Talarczyk-Noble for her support during this work.

\bibliography{CLTPapers}

\end{document}